\documentclass[12pt,reqno]{amsart}%
\usepackage{amsmath,amsthm,amscd,amsthm,upref,indentfirst}
\usepackage{amsfonts,mathrsfs}
\usepackage{amssymb,amsbsy,bm}
\usepackage{graphicx}
\usepackage{amssymb}%
\usepackage{times}
\usepackage[usenames]{color}
\theoremstyle{plain}
\newtheorem{theorem}{Theorem}
\newtheorem{assertion}[theorem]{Assertion}

\newtheorem{proposition}[theorem]{Proposition}

\theoremstyle{definition}
\newtheorem{definition}[theorem]{Definition}

\newtheorem{conjecture}[theorem]{Conjecture}
\newtheorem{corollary}[theorem]{Corollary}

\theoremstyle{remark}
\newtheorem{remark}[theorem]{Remark}

\newtheorem{example}[theorem]{Example}
\numberwithin{equation}{section}
\numberwithin{theorem}{section}
\usepackage[scaled=0.86]{helvet}
\renewcommand{\mathfrak}{\textsc}
\renewcommand{\mathbf}{\bm}
\renewcommand{\textit}{\textsf}
\numberwithin{equation}{section}
\numberwithin{theorem}{section}
\def\thetheorem{{\color{Brown}\bf\arabic{section}.\arabic{theorem}}}
\def\theequation{{\color{Magenta}{\sf\arabic{section}.\arabic{equation}}}}

\usepackage[normalem]{ulem}
\usepackage[square,comma]{natbib}
\usepackage{multirow}
\usepackage{epigraph}

\usepackage[overload]{textcase}
\author{Steven Duplij}
\date{\textbf{July 3, 2017}}
\address{\noindent Mathematisches Institute\\
Universit\"at M\"unster\newline
\mbox{} \hskip 10pt Einsteinstr. 62\\
D-48149 M\"unster,
Deutschland}
\email{duplijs@math.uni-muenster.de, sduplij@gmail.com}
\urladdr{http://homepages.spa.umn.edu/\~{}duplij}
\title[Polyadic integer numbers and finite  $\left( \MakeLowercase{m,n} \right)$-fields
]{\textbf{Polyadic integer numbers and finite  $\left( \MakeLowercase{m,n} \right)$-fields
}}

\begin{document}

\mbox{}

\begin{abstract}
\noindent 
The polyadic integer numbers, which form a polyadic ring, 
are representatives of a fixed congruence class. 
The basics of polyadic arithmetic are presented: prime polyadic numbers, 
the polyadic Euler function, polyadic division with a remainder, etc. are introduced. 
Secondary congruence classes of polyadic integer numbers, which become 
ordinary residue classes in the "binary limit", and the corresponding finite polyadic rings are defined. 
Polyadic versions of (prime) finite fields are introduced. 
These can be zeroless, zeroless and nonunital, or have several units; 
it is even possible for all of their elements to be units. 
There exist non-isomorphic finite polyadic fields of the same arity shape and order. 
None of the above situations is possible in the binary case. It is conjectured that 
a finite polyadic field should contain a certain canonical prime polyadic field, defined here, 
as a minimal finite subfield, which can be considered as a polyadic analogue of $GF\left( p \right)$.

\end{abstract}
\maketitle

\thispagestyle{empty}

\bigskip

\tableofcontents

\newpage

\mbox{}

\bigskip

\epigraph{
\emph{
\noindent In Galois Fields, full of flowers\\
primitive elements dance for hours\\
climbing sequentially through the trees\\
and shouting occasional parities.}}{S.B. Weinstein,
IEEE Trans. 1971}

\mbox{}

\bigskip

\section*{\textsc{Introduction}}

The theory of finite fields \cite{lid/nie} plays a very important role. From
one side, it acts as a \textquotedblleft gluing particle\textquotedblright%
\ connecting algebra, combinatorics and number theory (see, e.g.
\cite{mul/pan}), and from another it has numerous applications to
\textquotedblleft reality\textquotedblright: in coding theory, cryptography
and computer science \cite{men/oor/van}. Therefore, any generalization or
variation of its initial statements can lead to interesting and useful
consequences for both of the above. There are two principal peculiarities of
finite fields: 1) Uniqueness - they can have only special numbers of elements
(the order is any power of a prime integer $p^{r}$) and this fully determines
them, in that all finite fields of the same order are isomorphic; 2) Existence
of their \textquotedblleft minimal\textquotedblright\ (prime) finite subfield
of order $p$, which is isomorphic to the congruence class of integers
$\mathbb{Z}\diagup p\mathbb{Z}$. Investigation of the latter is a bridge to
the study of all finite fields, since they act as building blocks of the
extended (that is, all) finite fields.

We propose a special - polyadic - version of the (prime) finite fields in such
a way that, instead of the binary ring of integers $\mathbb{Z}$, we consider a
polyadic ring. The concept of the polyadic integer numbers $\mathbb{Z}%
_{\left(  m,n\right)  }$ as representatives of a fixed congruence class, which
form the $\left(  m,n\right)  $-ring (with $m$-ary addition and $n$-ary
multiplication), was introduced in \cite{dup2017}. Here we analyze
$\mathbb{Z}_{\left(  m,n\right)  }$ in more detail, by developing elements of
a polyadic analog of binary arithmetic: polyadic prime numbers, polyadic
division with a remainder, the polyadic Euler totient function, etc. ... It is
important to stress that the polyadic integer numbers are special variables
(we use superscripts for them) which in general have no connection with
ordinary integers (despite the similar notation used in computations), because
the former satisfy different relations, and coincide with the latter in the
binary case only. Next we will define new secondary congruence classes and
the corresponding finite $\left(  m,n\right)  $-rings $\mathbb{Z}_{\left(
m,n\right)  }\left(  q\right)  $ of polyadic integer numbers, which give
$\mathbb{Z}\diagup q\mathbb{Z}$ in the ``binary limit''. The conditions under
which these rings become fields are given, and the corresponding
\textquotedblleft abstract\textquotedblright\ polyadic fields are defined and
classified using their idempotence polyadic order. They have unusual
properties, and can be zeroless, zeroless-nonunital or have several units, and
it is even possible for all elements to be units. The subgroup structure of
their (cyclic) multiplicative finite $n$-ary group is analyzed in detail. For
some zeroless finite polyadic fields their multiplicative $n$-ary group is a
non-intersecting union of subgroups. It is shown that there exist
non-isomorphic finite polyadic fields of the same arity shape and order. None
of the above situations is possible in the binary case.

Some general properties of polyadic rings and fields were given in
\cite{cro2,lee/but,pop/pop02,dup/wer}, but their concrete examples using
integers differ considerably from our construction here, and the latter leads
to so called nonderived (proper) versions which have not been considered before.

We conjecture that any $\left(  m,n\right)  $-field with $m>n$ contains as a subfield one
of the prime polyadic fields constructed here, which can be considered as a
polyadic analog of $GF\left(  p\right)  $.

\section{\textsc{Preliminaries}}

We use the notations and definitions from \cite{dup2012,dup2017} (see, also,
references therein). We recall (only for self-consistency) some important
elements and facts about polyadic rings, which will be needed below.

Informally, a polyadic $\left(  m,n\right)  $-ring is $\mathcal{R}%
_{m,n}=\left\langle R\mid\mathbf{\nu}_{m},\mathbf{\mu}_{n}\right\rangle $,
where $R$ is a set, equipped with $m$-ary addition $\mathbf{\nu}_{m}%
:R^{m}\rightarrow R$ and $n$-ary multiplication $\mathbf{\mu}_{n}%
:R^{n}\rightarrow R$ which are connected by the polyadic distributive law,
such that $\left\langle R\mid\mathbf{\nu}_{m}\right\rangle $ is a commutative
$m$-ary group and $\left\langle R\mid\mathbf{\mu}_{n}\right\rangle $ is a
semigroup. A \textit{commutative (cancellative) polyadic ring} has a
commutative (cancellative) $n$-ary multiplication $\mathbf{\mu}_{n}$. A
polyadic ring is called \textit{derived}, if $\mathbf{\nu}_{m}$ and
$\mathbf{\mu}_{n}$ are equivalent to a repetition of the binary addition and
multiplication, while $\left\langle R\mid\mathbf{+}\right\rangle $ and
$\left\langle R\mid\mathbf{\cdot}\right\rangle $ are commutative (binary)
group and semigroup respectively. If only one operation $\mathbf{\nu}_{m}$ (or
$\mathbf{\mu}_{n}$) has this property, we call such a $\mathcal{R}_{m,n}$
\textit{additively} (or \textit{multiplicatively}) derived
(\textit{half-derived}).

In distinction to binary rings, an $n$-\textit{admissible} \textquotedblleft
length of word $\left(  \mathbf{x}\right)  $\textquotedblright\ should be
congruent to $1\operatorname{mod}\left(  n-1\right)  $, containing $\ell_{\mu
}\left(  n-1\right)  +1$ elements ($\ell_{\mu}$ is a \textquotedblleft number
of multiplications\textquotedblright) $\mathbf{\mu}_{n}^{\left(  \ell_{\mu
}\right)  }\left[  \mathbf{x}\right]  $ ($\mathbf{x}\in R^{\ell_{\mu}\left(
n-1\right)  +1}$), so called $\left(  \ell_{\mu}\left(  n-1\right)  +1\right)
$-ads, or \textit{polyads}. An $m$-\textit{admissible} \textquotedblleft
quantity of words $\left(  \mathbf{y}\right)  $\textquotedblright\ in a
polyadic \textquotedblleft sum\textquotedblright\ has to be congruent to
$1\operatorname{mod}\left(  m-1\right)  $, i.e. consisting of $\ell_{\nu
}\left(  m-1\right)  +1$ summands ($\ell_{\nu}$ is a \textquotedblleft number
of additions\textquotedblright) $\mathbf{\nu}_{m}^{\left(  \ell_{\nu}\right)
}\left[  \mathbf{y}\right]  $ ($\mathbf{y}\in R^{\ell_{\nu}\left(  m-1\right)
+1}$). Therefore, a straightforward \textquotedblleft
polyadization\textquotedblright\ of any binary expression ($m=n=2$) can be
introduced as follows: substitute the number of multipliers $\ell_{\mu
}+1\rightarrow\ell_{\mu}\left(  n-1\right)  +1$ and number of summands
$\ell_{\nu}+1\rightarrow\ell_{\nu}\left(  m-1\right)  +1$, respectively.

An example of \textquotedblleft trivial polyadization\textquotedblright\ is
the simplest $\left(  m,n\right)  $-ring derived from the ring of integers
$\mathbb{Z}$ as the set of $\ell_{\nu}\left(  m-1\right)  +1$
\textquotedblleft sums\textquotedblright\ of $n$-admissible $\left(  \ell
_{\mu}\left(  n-1\right)  +1\right)  $-ads $\left(  \mathbf{x}\right)  $,
where $\mathbf{x}\in\mathbb{Z}^{\ell_{\mu}\left(  n-1\right)  +1}$
\cite{lee/but}.

The additive $m$-ary \textit{ polyadic power} and the multiplicative
$n$-\textit{ary polyadic power} are defined by (inside polyadic products we
denote repeated entries by $\overset{k}{\overbrace{x,\ldots,x}}$ as $x^{k}$)%
\begin{equation}
x^{\left\langle \ell_{\nu}\right\rangle _{+m}}=\mathbf{\nu}_{m}^{\left(
\ell_{\nu}\right)  }\left[  x^{\ell_{\nu}\left(  m-1\right)  +1}\right]
,\ \ \ \ x^{\left\langle \ell_{\mu}\right\rangle _{\times n}}=\mathbf{\mu}%
_{n}^{\left(  \ell_{\mu}\right)  }\left[  x^{\ell_{\mu}\left(  n-1\right)
+1}\right]  ,\ \ x\in R, \label{xp}%
\end{equation}
such that the polyadic powers and ordinary powers differ by one:
$x^{\left\langle \ell_{\nu}\right\rangle _{+2}}=x^{\ell_{\nu}+1}$,
$x^{\left\langle \ell_{\mu}\right\rangle _{\times2}}=x^{\ell_{\mu}+1}$.

The \textit{polyadic} \textit{idempotents} in $\mathcal{R}_{m,n}$ satisfy%
\begin{equation}
x^{\left\langle \ell_{\nu}\right\rangle _{+m}}=x,\ \ \ \ x^{\left\langle
\ell_{\mu}\right\rangle _{\times n}}=x, \label{xx}%
\end{equation}
and are called the \textit{additive }$\ell_{\nu}$\textit{-idempotent} and the
\textit{multiplicative }$\ell_{\mu}$\textit{-idempotent}, respectively.

The additive $1$-idempotent, the \textit{zero} $z\in R$, is (if it exists)
defined by%
\begin{equation}
\mathbf{\nu}_{m}\left[  \mathbf{x},z\right]  =z,\ \ \ \forall\mathbf{x}\in
R^{m-1}. \label{z}%
\end{equation}

An element $x\in R$ is called (polyadic) \textit{nilpotent}, if
$x^{\left\langle 1\right\rangle _{+m}}=z$, and all higher powers of a
nilpotent element are nilpotent, as follows from (\ref{z}) and associativity.

The \textit{unit} $e$ of $\mathcal{R}_{m,n}$ is a multiplicative
$1$-idempotent which is defined (if it exists) as
\begin{equation}
\mathbf{\mu}_{n}\left[  e^{n-1},x\right]  =x,\ \ \ \forall x\in R, \label{e}%
\end{equation}
where (in case of a noncommutative polyadic ring) $x$ can be on any place. An
element $x\in R$ is called a (polyadic) $\ell_{\mu}$\textit{-reflection}, if
$\ x^{\left\langle \ell_{\mu}\right\rangle _{\times n}}=e$ (multiplicative
analog of a nilpotent element).

Polyadic rings with zero or unit(s) are called additively or multiplicatively
\textsl{half-derived}, and derived rings have a zero and unit(s)
simultaneously. There are polyadic rings which have no unit and no zero, or
with several units and no zero, or where all elements are units. But if a zero
exists, it is unique. If a polyadic ring contains no unit and no zero, we call
it a \textit{zeroless nonunital polyadic ring}. It is obvious that zeroless
nonunital rings can contain other idempotents of higher polyadic powers.

So, in polyadic rings (including the zeroless nonunital ones) invertibility
can be governed in a way which is not connected with unit and zero elements.
For a fixed element $x\in R$ its \textit{additive} \textit{querelement}
$\tilde{x}$ and \textit{multiplicative} \textit{querelement} $\bar{x}$ are
defined by%
\begin{equation}
\mathbf{\nu}_{m}\left[  x^{m-1},\tilde{x}\right]  =x,\ \ \ \ \ \ \mathbf{\mu
}_{n}\left[  x^{n-1},\bar{x}\right]  =x, \label{nm}%
\end{equation}
where in the second equation, if the $n$-ary multiplication $\mathbf{\mu}_{n}$
is noncommutative, $\bar{x}$ can be on any place. Because $\left\langle
R\mid\mathbf{\nu}_{m}\right\rangle $ is a commutative group, each $x\in R$ has
its additive querelement $\tilde{x}$ (and is \textit{querable} or
\textquotedblleft polyadically invertible\textquotedblright). The $n$-ary
semigroup $\left\langle R\mid\mathbf{\mu}_{n}\right\rangle $ can have no
multiplicatively querable elements at all. However, if every $x\in R$ has its
unique querelement, then $\left\langle R\mid\mathbf{\mu}_{n}\right\rangle $ is
an $n$-ary group. Obviously, that $n$-ary group cannot have nilpotent
elements, but can have $\ell_{\mu}$-reflections. Denote $R^{\ast}%
=R\setminus\left\{  z\right\}  $, if the zero $z$ exists. If $\left\langle
R^{\ast}\mid\mathbf{\mu}_{n}\right\rangle $ is the $n$-ary group, then
$\mathcal{R}_{m,n}$ is a $\left(  m,n\right)  $-\textit{division}
\textit{ring}.

\begin{definition}
\label{def-fmn}A commutative $\left(  m,n\right)  $-division ring
$\mathcal{R}_{m,n}$ is a $\left(  m,n\right)  $-\textit{field }$\mathcal{F}%
_{m,n}$.
\end{definition}

The simplest example of a $\left(  m,n\right)  $-field derived from
$\mathbb{R}$ is the set of $\ell_{\nu}\left(  m-1\right)  +1$
\textquotedblleft sums\textquotedblright\ of admissible $\left(  \ell_{\mu
}\left(  n-1\right)  +1\right)  $-ads $\left(  \mathbf{x}\right)  $, where
$\mathbf{x}\in\mathbb{R}^{\ell_{\mu}\left(  n-1\right)  +1}$. Some
\textsl{nonderived} $\left(  m,n\right)  $-fields are in

\begin{example}
\textbf{a)} The set $i\mathbb{R}$ with $i^{2}=-1$ is a $\left(  2,3\right)
$-field with a zero and no unit (operations are made in $\mathbb{C}$), but the
multiplicative querelement of $ix$ is $-i\diagup x$ ($x\neq0$).

\textbf{b}) The set of fractions $\left\{  ix/y\mid x,y\in\mathbb{Z}%
^{odd},\ i^{2}=-1\right\}  $ is a $\left(  3,3\right)  $-field with no zero
and no unit (operations are in $\mathbb{C}$), while the additive and
multiplicative querelements of $ix/y$ are $-ix/y$ and $-iy/x$, respectively.

\textbf{c)} The set of antidiagonal $2\times2$ matrices over $\mathbb{R}$ is a
$\left(  2,3\right)  $-field with zero $z=\left(
\begin{array}
[c]{cc}%
0 & 0\\
0 & 0
\end{array}
\right)  $ and \textsl{two} units $e=\pm\left(
\begin{array}
[c]{cc}%
0 & 1\\
1 & 0
\end{array}
\right)  $, but the unique querelement of $\left(
\begin{array}
[c]{cc}%
0 & x\\
y & 0
\end{array}
\right)  $ is $\left(
\begin{array}
[c]{cc}%
0 & 1/y\\
1/x & 0
\end{array}
\right)  $.
\end{example}

\section{\textsc{Ring of polyadic integer numbers}}

Recall the notion of the ring of polyadic integer numbers $\mathbb{Z}_{\left(
m,n\right)  }$ which was introduced in \cite{dup2017}, where its difference
from the $\left(  m,n\right)  $-ring of integers from \cite{lee/but} was outlined.

Let us consider a congruence class (residue class) of an integer $a$ modulo
$b$%
\begin{equation}
\left[  \left[  a\right]  \right]  _{b}=\left\{  \left\{  a+bk\right\}  \mid
k\in\mathbb{Z},\ \ a\in\mathbb{Z}_{+},\ b\in\mathbb{N},\ 0\leq a\leq
b-1\right\}  . \label{ab}%
\end{equation}
We denote a representative element by $x_{k}=x_{k}^{\left[  a,b\right]
}=a+bk$, where obviously $\left\{  x_{k}\right\}  $ is an infinite set.

\subsection{External and internal operations for congruence classes}

Informally, there are two ways to equip (\ref{ab}) with operations:

\begin{enumerate}
\item The \textquotedblleft External\textquotedblright\ way is to define
(\textsl{binary}) operations between the congruence classes. Let us define on
the finite underlying set of $b$ congruence classes $\left\{  \left[  \left[
a\right]  \right]  _{b}\right\}  $, $a=0,1,\ldots,b-1$ the following new
binary operations (here, if $b$ is fixed, and we denote the binary class
representative by an integer with \textsl{one} prime $\left[  \left[
a\right]  \right]  _{b}\equiv a^{\prime}$, as well as the corresponding binary
operations $+^{\prime},\cdot^{\prime}$ between classes)%
\begin{align}
a_{1}^{\prime}+^{\prime}a_{2}^{\prime}  &  =\left(  a_{1}+a_{2}\right)
^{\prime},\label{a1}\\
a_{1}^{\prime}\cdot^{\prime}a_{2}^{\prime}  &  =\left(  a_{1}a_{2}\right)
^{\prime}. \label{a2}%
\end{align}
Then, the binary \textit{residue class ring} is defined by%
\begin{equation}
\mathbb{Z}\diagup b\mathbb{Z=}\left\{  \left\{  a^{\prime}\right\}
\mid+^{\prime},\cdot^{\prime},0^{\prime},1^{\prime}\right\}  . \label{zb}%
\end{equation}
In the case of prime $b=p$, the ring $\mathbb{Z}\diagup p\mathbb{Z}$ becomes a
\textsl{binary} finite field having $p$ elements.

\item The \textquotedblleft Internal\textquotedblright\ way is to introduce
(\textsl{polyadic}) operations inside a given class $\left[  \left[  a\right]
\right]  _{b}$ (with \textsl{both} $a$ and $b$ fixed). We introduce the
commutative $m$-ary addition and commutative $n$-ary multiplication of
representatives $x_{k_{i}}$ of the fixed congruence class by%
\begin{align}
\nu_{m}\left[  x_{k_{1}},x_{k_{2}},\ldots,x_{k_{m}}\right]   &  =x_{k_{1}%
}+x_{k_{2}}+\ldots+x_{k_{m}},\label{nu}\\
\mu_{n}\left[  x_{k_{1}},x_{k_{2}},\ldots,x_{k_{n}}\right]   &  =x_{k_{1}%
}x_{k_{2}}\ldots x_{k_{n}},\ \ \ x_{k_{i}}\in\left[  \left[  a\right]
\right]  _{b},\ k_{i}\in\mathbb{Z}. \label{mu}%
\end{align}

In general, the binary sums $x_{k_{1}}+x_{k_{2}}$ and products $x_{k_{1}%
}x_{k_{2}}$ are not in $\left[  \left[  a\right]  \right]  _{b}$.

\begin{proposition}
[\cite{dup2017}]The polyadic operations $\nu_{m}$ and $\mu_{n}$ become closed
in $\left[  \left[  a\right]  \right]  _{b}$, if the arities $\left(
m,n\right)  $ have the \textsl{minimal} values satisfying%
\begin{align}
ma  &  \equiv a\left(  \operatorname{mod}b\right)  ,\label{maa}\\
a^{n}  &  \equiv a\left(  \operatorname{mod}b\right)  . \label{ana}%
\end{align}

\end{proposition}

Polyadic distributivity is inherited from that of $\mathbb{Z}$, and therefore
we have

\begin{definition}
[\cite{dup2017}]The congruence class $\left[  \left[  a\right]  \right]  _{b}$
equipped with a structure of \textsl{nonderived} \textsl{infinite} commutative
polyadic ring is called a $\left(  m,n\right)  $-\textit{ring of polyadic
integer numbers}%
\begin{equation}
\mathbb{Z}_{\left(  m,n\right)  }\equiv\mathbb{Z}_{\left(  m,n\right)
}^{\left[  a,b\right]  }=\left\{  \left[  \left[  a\right]  \right]  _{b}%
\mid\nu_{m},\mu_{n}\right\}  . \label{za}%
\end{equation}
Obviously, $\mathbb{Z}_{\left(  m,n\right)  }$ (as in the binary case) cannot
become a polyadic field with any choice of parameters.
\end{definition}
\end{enumerate}

\begin{example}
\label{exam-34}In the residue class%
\begin{equation}
\left[  \left[  3\right]  \right]  _{4}=\left\{  \ldots
-25,-21,-17,-13,-9,-5,-1,3,7,11,15,19,23,27,31,35,39\ldots\right\}  \label{34}%
\end{equation}
we can add \textsl{only} $4\ell_{\nu}+1$ representatives and multiply
$2\ell_{\mu}+1$ representatives ($\ell_{\nu},\ell_{\mu}$ are \textquotedblleft
numbers\textquotedblright\ of $m$-ary additions and $n$-ary multiplications
respectively) to retain the same class, e.g., take $\ell_{\nu}=2$, $\ell_{\mu
}=3$ to get $\left(  7+11+15+19+23\right)  -5-9-13-1=\allowbreak47\in\left[
\left[  3\right]  \right]  _{4}$, $\left(  \left(  7\cdot3\cdot11\right)
\cdot19\cdot15\right)  \cdot31\cdot27=\allowbreak55\,103\,895\in\left[
\left[  3\right]  \right]  _{4}$. Obviously, we cannot add and multiply
arbitrary quantities of numbers in $\left[  \left[  3\right]  \right]  _{4}$,
only the \textsl{admissible} ones. This means that $\left[  \left[  3\right]
\right]  _{4}$ is the polyadic $\left(  5,3\right)  $-ring $\mathbb{Z}%
_{\left(  5,3\right)  }=\mathbb{Z}_{\left(  5,3\right)  }^{\left[  3,4\right]
}$.
\end{example}

\begin{remark}
After imposing the operations (\ref{nu})--(\ref{mu}) the representatives
$x_{k}^{\left[  a,b\right]  }$ become abstract variables (elements of the
corresponding $\left(  m,n\right)  $-ring or polyadic integer numbers) which
are not ordinary integers (forming a $\left(  2,2\right)  $-ring), but have
the latter as their \textquotedblleft binary limit\textquotedblright. So in
computations the integral numbers (denoting representatives) should carry
their arity shape $\left(  m,n\right)  $ as additional indices. Indeed, the
representative, e.g. $3=3_{\left(  5,3\right)  }\in\mathbb{Z}_{\left(
5,3\right)  }^{\left[  3,4\right]  }$ is different from $3=3_{\left(
3,2\right)  }\in\mathbb{Z}_{\left(  3,2\right)  }^{\left[  1,2\right]  }$,
i.e. properly speaking $3_{\left(  5,3\right)  }\neq3_{\left(  3,2\right)  }$,
since their operations (multiplication and addition) are \textsl{different},
because they belong to \textsl{different} polyadic rings, $\mathbb{Z}_{\left(
5,3\right)  }^{\left[  3,4\right]  }$ and $\mathbb{Z}_{\left(  3,2\right)
}^{\left[  1,2\right]  }$, respectively. For conciseness, we omit the indices
$\left(  m,n\right)  $, if their value is clear from the context.
\end{remark}

Thus, at first sight it seems that one can obtain a polyadic field only in the
\textquotedblleft external\textquotedblright\ way, i.e. using the
\textquotedblleft trivial polyadization\textquotedblright\ of the binary
finite field $\mathbb{Z}\diagup p\mathbb{Z}$ (just a repetition of the binary
group operations).This leads to the \textsl{derived} polyadic finite fields,
which have a very simple structure, in which the admissible binary sums and
binary products of the congruence classes are used \cite{lee/but}. However, in
the next section we propose a new approach, and thereby construct the
\textsl{nonderived} finite $\left(  m,n\right)  $-fields of polyadic integer
numbers $\mathbb{Z}_{\left(  m,n\right)  }$.

\begin{remark}
If $n=b=p$ is prime, then (\ref{ana}) is valid for any $a\in\mathbb{N}$, which
is another formulation of Fermat's little theorem.
\end{remark}

\subsection{Prime polyadic integer numbers}

Let us introduce a polyadic analog of prime numbers in $\mathbb{Z}_{\left(
m,n\right)  }$. First we need

\begin{definition}
A \textit{polyadically composite (reducible) number} is $x_{k}\in
\mathbb{Z}_{\left(  m,n\right)  }$, such that the expansion%
\begin{equation}
x_{k}=\mu_{n}^{\left(  \ell\right)  }\left[  x_{k_{1}},x_{k_{2}}%
,,x_{k_{\ell\left(  n-1\right)  +1}}\right]  ,\ \ \ x_{k_{i}}\in
\mathbb{Z}_{\left(  m,n\right)  }, \label{x0}%
\end{equation}
is unique, where $\ell$ is a number of $n$-ary multiplications, and there
exist at least one $x_{k_{i}}\neq x_{k}$ and $x_{k_{i}}\neq e$ (i.e. is not
equal to unit of $\mathbb{Z}_{\left(  m,n\right)  }$, if it exists). Denote
the set of such numbers $\left\{  x_{k_{i}}\right\}  =\mathbb{D}\left(
x_{k}\right)  $ which is called the \textit{composition set} of $x_{k}$.
\end{definition}

\begin{definition}
An \textit{irreducible polyadic number} is $x_{k}\in\mathbb{Z}_{\left(
m,n\right)  }$ cannot be expressed as any (long) polyadic product (\ref{x0}).
\end{definition}

\begin{proposition}
In the polyadic ring $\mathbb{Z}_{\left(  m,n\right)  }^{\left[  a,b\right]
}$ without the unit the elements satisfying%
\begin{equation}
-\left\vert a-b\right\vert ^{n}<x_{k}<\left\vert a-b\right\vert ^{n}%
\end{equation}
are irreducible.
\end{proposition}

\begin{proof}
Since $0\leq a\leq b-1$, the minimal absolute value of an element $x_{k}%
\in\mathbb{Z}_{\left(  m,n\right)  }^{\left[  a,b\right]  }$ is $\left\vert
a-b\right\vert $. The minimum of its $n$-ary product is $\left\vert
a-b\right\vert ^{n}$, and therefore smaller elements cannot be decomposed.
\end{proof}

\begin{example}
In the $\left(  6,5\right)  $-ring $\mathbb{Z}_{\left(  6,5\right)  }^{\left[
8,10\right]  }$ all polyadic integer numbers are even, and there is no unit,
and so they are binary composite%
\begin{equation}
\mathbb{Z}_{\left(  6,5\right)  }^{\left[  8,10\right]  }=\left\{
\ldots-72,-62,-52,-42,-32,-22,-12,-2,8,18,28,38,48,58,\ldots\right\}
\label{z810}%
\end{equation}
Nevertheless, the lowest elements, e.g. $\left\{  -22,-12,8,18,28\right\}  $,
are irreducible, while the smallest (by absolute value) polyadically composite
element is $\left(  -32\right)  =\mu_{5}\left[  \left(  -2\right)
^{5}\right]  $.
\end{example}

\begin{definition}
A range in which all elements are indecomposable is called a \textit{polyadic
irreducible gap}.
\end{definition}

\begin{remark}
We do not demand positivity, as in the binary case, because polyadic integer
numbers $\mathbb{Z}_{\left(  m,n\right)  }^{\left[  a,b\right]  }$ (\ref{za})
are \textquotedblleft symmetric\textquotedblright\ not with respect to $x=0$,
but under $x=x_{k=0}=a$.
\end{remark}

The polyadic analog of binary prime numbers plays an intermediate role between
composite and irreducible elements.

\begin{definition}
A\textit{ polyadic prime number} is $x_{k_{p}}\in\mathbb{Z}_{\left(
m,n\right)  }$, such that it obeys only the unique expansion%
\begin{equation}
x_{k_{p}}=\mu_{n}^{\left(  \ell\right)  }\left[  x_{k_{p}},e^{\ell\left(
n-1\right)  }\right]  , \label{xkp}%
\end{equation}
where $e$ a polyadic unit of $\mathbb{Z}_{\left(  m,n\right)  }$ (if exists).
\end{definition}

So, the polyadic prime numbers can appear only in those polyadic rings
$\mathbb{Z}_{\left(  m,n\right)  }^{\left[  a,b\right]  }$ which contain
units. In \cite{dup2017} (\textbf{Proposition 6.15}) it was shown that such
rings correspond to the \textit{limiting congruence classes} $\left[  \left[
1\right]  \right]  _{b}$ and $\left[  \left[  b-1\right]  \right]  _{b}$, and
indeed only for them can $a+bk=1\operatorname{mod}b$, and $e^{\ell\left(
n-1\right)  }$ can be a neutral sequence (for $e=1$ always, while for $e=-1$
only when $\ell\left(  n-1\right)  $ is even).

\begin{proposition}
\label{prop-lim}The prime polyadic numbers can exist only in the limiting
polyadic rings $\mathbb{Z}_{\left(  b+1,2\right)  }^{\left[  1,b\right]  }$
and $\mathbb{Z}_{\left(  b+1,3\right)  }^{\left[  b-1,b\right]  }$.
\end{proposition}

\begin{proof}
The equation $a+bk=1\operatorname{mod}b$ (for $0\leq a\leq b-1$) has two
solutions: $a=1$ and $a=b-1$ corresponding for two limiting congruence classes
$\left[  \left[  1\right]  \right]  _{b}$ and $\left[  \left[  b-1\right]
\right]  _{b}$, which correspond to%
\begin{align}
x_{k}^{+}  &  =bk+1,\label{xk1}\\
x_{k}^{-}  &  =b\left(  k+1\right)  -1,\ \ \ \ k\in\mathbb{Z}. \label{xk2}%
\end{align}
The parameters-to-arity mapping (\ref{psi}) fixes their multiplication arity
to $n=2$ and $n=3$ respectively, which gives manifestly%
\begin{align}
&  \mu_{2}\left[  x_{k_{1}}^{+},x_{k_{2}}^{+}\right]  =x_{k}^{+}%
,\ \ \ \ \ \ \ \ \ \ \ k=bk_{1}k_{2}+k_{1}+k_{2},\label{xxx}\\
&  \mu_{3}\left[  x_{k_{1}}^{-},x_{k_{2}}^{-},x_{k_{3}}^{-}\right]  =x_{k}%
^{-},\ \ \ \ k,k_{i}\in\mathbb{Z},\ \ \ i=1,2,3,\ \ \ \ \ \ b\in
\mathbb{N},\nonumber\\
&  k=b^{2}k_{1}k_{2}k_{3}+\left(  b-1\right)  \left[  b\left(  k_{1}%
k_{2}+k_{2}k_{3}+k_{1}k_{3}\right)  +\left(  b-1\right)  \left(  k_{1}%
+k_{2}+k_{3}\right)  +\left(  b-2\right)  \right]  . \label{xxxx}%
\end{align}
Therefore, for $\mathbb{Z}_{\left(  b+1,2\right)  }^{\left[  1,b\right]  }$ we
have the unit $e=x_{k=0}^{+}=1$ (which is obvious for the binary
multiplication), while in $\mathbb{Z}_{\left(  b+1,3\right)  }^{\left[
b-1,b\right]  }$ the unit is $e=x_{k=-1}^{-}=-1$ ($b\geq3$), and the sequence
is $\left(  e^{2}\right)  $ is evidently neutral.
\end{proof}

Denote the set of ordinary \textsl{binary prime numbers} in the interval
$1\leq k\leq k_{\max}$ by $\mathbb{P}\left(  k_{\max}\right)  $, $k_{\max}%
\in\mathbb{N}$. The set of \textsl{prime polyadic numbers} for the polyadic
ring $\mathbb{Z}_{\left(  m,n\right)  }^{\left[  a,b\right]  }$ in the
interval $x_{-k_{\max}}\leq x_{k}\leq x_{k_{\max}}$ is denoted by
$\mathbb{P}_{\left(  m,n\right)  }^{\left[  a,b\right]  }\left(  k_{\max
}\right)  $. Obviously, in the binary limit $\mathbb{P}_{\left(  2,2\right)
}^{\left[  0,1\right]  }\left(  k_{\max}\right)  =\mathbb{P}\left(  k_{\max
}\right)  \cup\left\{  -\mathbb{P}\left(  k_{\max}\right)  \right\}  $.
Nevertheless, prime polyadic numbers can be composite as binary numbers.

\begin{assertion}
The set of prime polyadic numbers in the interval $x_{-k_{\max}}\leq x_{k}\leq
x_{k_{\max}}$ for $\mathbb{Z}_{\left(  m,n\right)  }^{\left[  a,b\right]  }$
can contain composite binary numbers, i.e.%
\begin{equation}
\Delta\mathbb{P}_{\left(  m,n\right)  }^{\left[  a,b\right]  }\left(  k_{\max
}\right)  =\mathbb{P}_{\left(  m,n\right)  }^{\left[  a,b\right]  }\left(
k_{\max}\right)  \setminus\left\{  \left\{  \mathbb{P}\left(  x_{k_{\max}%
}\right)  \cup\left\{  \mathbb{P}\left(  -x_{-k_{\max}}\right)  \right\}
\right\}  \cap\left[  \left[  a\right]  \right]  _{b}\right\}  \neq
\varnothing. \label{pp}%
\end{equation}

\end{assertion}

\begin{definition}
\textbf{1)} The cardinality of the set of ordinary binary prime numbers
$\mathbb{P}\left(  k_{\max}\right)  $ is called a \textit{prime-counting
function} and denoted by $\pi\left(  k_{\max}\right)  =\left\vert
\mathbb{P}\left(  k_{\max}\right)  \right\vert $.

\textbf{2)} The cardinality of the set of prime polyadic numbers
$\mathbb{P}_{\left(  m,n\right)  }^{\left[  a,b\right]  }\left(  k_{\max
}\right)  $ is called a \textit{polyadic prime-counting function} and denoted
by%
\begin{equation}
\pi_{\left(  m,n\right)  }^{\left[  a,b\right]  }\left(  k_{\max}\right)
=\left\vert \mathbb{P}_{\left(  m,n\right)  }^{\left[  a,b\right]  }\left(
k_{\max}\right)  \right\vert .
\end{equation}

\end{definition}

\begin{example}
\textbf{1)} Consider $\mathbb{Z}_{\left(  45,3\right)  }^{\left[
43,44\right]  }$ and $k_{\max}=2$, then%
\begin{align}
\mathbb{P}_{\left(  45,3\right)  }^{\left[  43,44\right]  }\left(  2\right)
&  =\left\{  -45,-1,43,87,131\right\}  ,\\
\Delta\mathbb{P}_{\left(  45,3\right)  }^{\left[  43,44\right]  }\left(
2\right)   &  =\left\{  -45\right\}  ,\ \ \ \pi_{\left(  45,3\right)
}^{\left[  43,44\right]  }\left(  2\right)  =5.
\end{align}

\textbf{2)} For $\mathbb{Z}_{\left(  52,3\right)  }^{\left[  50,51\right]  }$
and $k_{\max}=5$ we have%
\begin{align}
\mathbb{P}_{\left(  52,3\right)  }^{\left[  50,51\right]  }\left(  5\right)
&  =\{-205,-154,-103,-52,-1,50,101,152,203,254,305\},\\
\Delta\mathbb{P}_{\left(  52,3\right)  }^{\left[  50,51\right]  }\left(
5\right)   &  =\{-205,-154,-52,50,152,203,254,305\},\ \ \pi_{\left(
52,3\right)  }^{\left[  50,51\right]  }\left(  5\right)  =11.
\end{align}

\end{example}

\begin{remark}
This happens because in $\mathbb{Z}_{\left(  m,n\right)  }^{\left[
a,b\right]  }$ the role of \textquotedblleft building blocks\textquotedblright%
\ (prime polyadic numbers) is played by those $x_{k}$ which cannot be
presented as a (long) ternary product of other polyadic integer numbers from
the same $\mathbb{Z}_{\left(  m,n\right)  }^{\left[  a,b\right]  }$ as in
(\ref{x0}), but which satisfy (\ref{xkp}) only. Nevertheless, such prime
polyadic numbers can be composite binary prime numbers.
\end{remark}

In general, for the limiting cases, in which polyadic prime numbers exist, we have

\begin{proposition}
\textbf{1)} In $\mathbb{Z}_{\left(  b+1,2\right)  }^{\left[  1,b\right]  }$
the \textquotedblleft smallest\textquotedblright\ polyadic integer numbers
satisfying%
\begin{align}
-b  &  <k_{p}<b+2,\\
1-b^{2}  &  <x_{k_{p}}<\left(  b+1\right)  ^{2}, \label{g2}%
\end{align}
are not decomposable, and therefore such $x_{k_{p}}\in\mathbb{Z}_{\left(
b+1,2\right)  }^{\left[  1,b\right]  }$ are \textsl{all} polyadic prime numbers.

\textbf{2)} For another limiting case $\mathbb{Z}_{\left(  b+1,3\right)
}^{\left[  b-1,b\right]  }$ the\ polyadic integer numbers satisfying%
\begin{align}
1-b  &  <k_{p}<b-1,\\
-\left(  b-1\right)  ^{2}  &  <x_{k_{p}}<b^{2}-1, \label{g3}%
\end{align}
are not ternary decomposable and so \textsl{all} such $x_{k_{p}}\in
\mathbb{Z}_{\left(  b+1,2\right)  }^{\left[  b-1,b\right]  }$ are polyadic
prime numbers.
\end{proposition}

\begin{proof}
This follows from determining the maximum of the negative values and the
minimum of the positive values of the functions $x_{k}^{+}$ and $x_{k}^{-}$ in
(\ref{xxx})--(\ref{xxxx}).
\end{proof}

\begin{definition}
The range in which all elements are polyadically prime numbers is called the
\textit{polyadic primes gap}, and for the two limiting cases it is given by
(\ref{g2}) and (\ref{g3}), respectively.
\end{definition}

For instance, in $\mathbb{Z}_{\left(  52,3\right)  }^{\left[  50,51\right]  }$
for the polyadic primes gap we have $-2500<x_{k_{p}}<2600$: all such polyadic
integer numbers are polyadically prime, but there are many composite binary
numbers among them.

In the same way we can introduce a polyadic analog of the \textit{Euler
(totient) function} which in the binary case counts the number of coprimes to
a given natural number. Denote the set of ordinary binary numbers $k>1$ which
are \textit{coprime} to $k_{\max}\in\mathbb{N}$ by $\mathbb{S}\left(  k_{\max
}\right)  $ (named totatives of $k_{\max}$). Then, the cardinality of
$\mathbb{S}\left(  k_{\max}\right)  $ is defined as Euler function
$\varphi\left(  k_{\max}\right)  =\left\vert \mathbb{S}\left(  k_{\max
}\right)  \right\vert $. Obviously, if $k_{\max}=p$ is prime, then
$\varphi\left(  p\right)  =p-1$. The notion of coprime numbers is based on the
divisors: the coprime numbers $k_{1}$ and $k_{2}$ have the greatest common
divisor $\gcd\left(  k_{1},k_{2}\right)  =1$. In the polyadic case it is not
so straightforward, and we need to start from the basic definitions.

First, we observe that in a (commutative) polyadic ring $\mathcal{R}%
_{m,n}=\left\{  R\mid\nu_{m},\mu_{n}\right\}  $ the analog of the division
operation is usually not defined \textsl{uniquely}, which makes it useless for
real applications. Indeed, $y$ divides $x$, where $x,y\in R$, if there exists
a sequence $\mathbf{z}\in R^{n-1}$ of length $\left(  n-1\right)  $, such that
$x=\mu_{n}\left[  y,\mathbf{z}\right]  $. To be consistent with the ordinary
integer numbers $\mathbb{Z}$, we demand in the polyadic number ring
$\mathbb{Z}_{\left(  m,n\right)  }$: \textbf{1)} \textsl{Uniqueness} of the
result; \textbf{2)} i.e. only \textsl{one} polyadic number (not a sequence) as
the result. This naturally leads to

\begin{definition}
A polyadic number (quotient) $x_{k_{2}}$ \textit{polyadically divides} a
polyadic number (dividend) $x_{k_{1}}$, if there exists $x_{k_{q}}:=x_{k_{1}%
}\div_{p}x_{k_{2}}$, called the (unique) \textit{result of division}, such
that%
\begin{equation}
x_{k_{1}}=\mu_{n}\left[  x_{k_{2}},\left(  x_{k_{q}}\right)  ^{n-1}\right]
,\ \ \ \ \ \ \ \ x_{k_{1}},x_{k_{2}},x_{k_{q}}\in\mathbb{Z}_{\left(
m,n\right)  }. \label{xq}%
\end{equation}

\end{definition}

\begin{remark}
For polyadic prime numbers (\ref{xkp}) the only possibility for the quotient
is $x_{k_{2}}=x_{k_{1}}$ such that $x_{k_{1}}=\mu_{n}\left[  x_{k_{1}},\left(
e\right)  ^{n-1}\right]  $ or $x_{k_{1}}\div_{p}x_{k_{1}}=e$, where $e$ is the
unit of $\mathbb{Z}_{\left(  m,n\right)  }$.
\end{remark}

\begin{assertion}
Polyadic division is distributive from the left%
\begin{equation}
\nu_{m}\left[  x_{k_{1}},x_{k_{2}},\ldots,x_{k_{m}}\right]  \div_{p}x_{k}%
=\nu_{m}\left[  \left(  x_{k_{1}}\div_{p}x_{k}\right)  ,\left(  x_{k_{2}}%
\div_{p}x_{k}\right)  ,\ldots,\left(  x_{k_{m}}\div_{p}x_{k}\right)  \right]
,
\end{equation}
but not distributive from the right.
\end{assertion}

\begin{proof}
This follows from the polyadic distributivity in the $\left(  m,n\right)
$-ring $\mathbb{Z}_{\left(  m,n\right)  }$.
\end{proof}

\begin{example}
\textbf{1)} In the polyadic ring $\mathbb{Z}_{\left(  10,4\right)  }^{\left[
4,9\right]  }$ we have uniquely $x_{28}\div_{p}x_{4}=x_{4}$ or $256\div
_{p}4=4$.

\textbf{2)} For the limiting ring $\mathbb{Z}_{\left(  5,3\right)  }^{\left[
3,4\right]  }$, we find $x_{43}\div_{p}x_{1}=x_{-2}$ or $175\div_{p}7=-5$.
\end{example}

In the same way we define a polyadic analog of division with a remainder.

\begin{definition}
A \textit{polyadic division with a remainder} is defined, if for a polyadic
dividend $x_{k_{1}}$ and divisor $x_{k_{2}}$ there exists a \textit{polyadic
remainder} $x_{k_{r}}$ such that%
\begin{equation}
x_{k_{1}}=\nu_{m}\left[  \mu_{n}\left[  x_{k_{2}},\left(  x_{k_{q}}\right)
^{n-1}\right]  ,\left(  x_{k_{r}}\right)  ^{m-1}\right]
,\ \ \ \ \ \ \ \ x_{k_{1}},x_{k_{2}},x_{k_{q}},x_{k_{r}}\in\mathbb{Z}_{\left(
m,n\right)  }, \label{xqr}%
\end{equation}
which is denoted by $x_{k_{r}}=x_{k_{1}}\operatorname{mod}_{p}x_{k_{2}}$, and
(\ref{xqr}) can be presented in the following binary form%
\begin{equation}
x_{k_{1}}=\left(  x_{k_{2}}\boxdot_{p}x_{k_{q}}\right)  \boxplus_{p}x_{k_{r}}.
\end{equation}

\end{definition}

The distributivity of these operations is governed by distributivity in the
polyadic ring $\mathbb{Z}_{\left(  m,n\right)  }$.

\begin{example}
In the polyadic ring $\mathbb{Z}_{\left(  6,5\right)  }^{\left[  8,10\right]
}$ we can have different divisions for the same dividend as $38=\left(
\left(  -22\right)  \boxdot_{p}\left(  -2\right)  \right)  \boxplus
_{p}78=\left(  \left(  -92\right)  \boxdot_{p}\left(  -2\right)  \right)
\boxplus_{p}238$.
\end{example}

Secondly, because divisibility in the polyadic case is not symmetric with
respect to dividend and divisor (\ref{xq}), we define polyadically coprime
numbers using the definition of compositeness (\ref{x0}).

\begin{definition}
The $s$ polyadic integer numbers $x_{k_{1}},\ldots,x_{k_{s}}\in\mathbb{Z}%
_{\left(  m,n\right)  }$ are \textit{polyadically coprime}, if their
composition sets do not intersect $\mathbb{D}\left(  x_{k_{1}}\right)
\cap\mathbb{D}\left(  x_{k_{1}}\right)  \cap\ldots\cap\mathbb{D}\left(
x_{k_{s}}\right)  =\varnothing$.
\end{definition}

It is important that this definition does not imply the existence of a unit in
$\mathbb{Z}_{\left(  m,n\right)  }$, as opposed to the definition of the
polyadic prime numbers (\ref{xkp}) in which the availability of a unit is crucial.

\begin{assertion}
Polyadically coprime numbers can exist in any polyadic ring $\mathbb{Z}%
_{\left(  m,n\right)  }$, and \textsl{not} only in the limiting cases with
unit (see \textbf{Proposition \ref{prop-lim}}).
\end{assertion}

\begin{corollary}
All elements in the polyadic irreducible gap are polyadically coprime.
\end{corollary}

\begin{example}
In $\mathbb{Z}_{\left(  6,5\right)  }^{\left[  8,10\right]  }$ (\ref{z810})
the polyadic integer numbers $\left(  -32\right)  =\mu_{5}\left[  \left(
-2\right)  ^{5}\right]  $ and $32768=\mu_{5}\left[  \left(  8\right)
^{5}\right]  $ are both composed, but polyadically coprime, because the
composition sets $\mathbb{D}\left(  -32\right)  =\left\{  -2\right\}  $ and
$\mathbb{D}\left(  32768\right)  =\left\{  8\right\}  $ do not intersect.
Alternatively, not polyadically coprime numbers here are, e.g., $\left(
-3072\right)  =\mu_{5}\left[  \left(  -12\right)  ,\left(  -2\right)
^{2},\left(  8\right)  ^{2}\right]  $ and $\left(  -64512\right)  =\mu
_{5}\left[  -2,\left(  8\right)  ^{2},18,28\right]  $, because $\mathbb{D}%
\left(  -3072\right)  \cap\mathbb{D}\left(  -64512\right)  =\left\{
-2,8\right\}  \neq\varnothing$. We cannot multiply these two numbers, because
the arity of multiplication is $5$.
\end{example}

\begin{remark}
\label{rem-pol}For polyadic integer numbers $\mathbb{Z}_{\left(  m,n\right)
}$ we cannot measure the property \textquotedblleft to be
coprime\textquotedblright\ in terms of a single element, as in binary case, by
their $\gcd$, because the $n$-ary multiplication is only allowed for
admissible sequences. Therefore, we need to consider the intersection of the
composition sets.
\end{remark}

In the polyadic ring $\mathbb{Z}_{\left(  m,n\right)  }^{\left[  a,b\right]
}$ for a given $k_{\max}\in\mathbb{Z}_{+}$ we denote by $\mathbb{S}_{\left(
m,n\right)  }^{\left[  a,b\right]  }\left(  k_{\max}\right)  $ the set of
polyadic prime numbers $x_{k}\in\mathbb{Z}_{\left(  m,n\right)  }^{\left[
a,b\right]  }$ which are polyadically coprime to $x_{k_{\max}}$ and
$x_{-k_{\max}}$ in the open interval $x_{-k_{\max}}<x_{k}<x_{k_{\max}}$ (we
assume that, if two numbers are coprime, then their opposite numbers are also
coprime). For the binary limit, obviously, $\mathbb{S}_{\left(  2,2\right)
}^{\left[  0,1\right]  }\left(  k_{\max}\right)  =\mathbb{S}\left(  k_{\max
}\right)  \cup\left\{  -\mathbb{S}\left(  k_{\max}\right)  \right\}  $.

\begin{definition}
The cardinality of $\mathbb{S}_{\left(  m,n\right)  }^{\left[  a,b\right]
}\left(  k_{\max}\right)  $ is called a \textit{polyadic Euler function}
denoted by%
\begin{equation}
\varphi_{\left(  m,n\right)  }^{\left[  a,b\right]  }\left(  k_{\max}\right)
=\left\vert \mathbb{S}_{\left(  m,n\right)  }^{\left[  a,b\right]  }\left(
k_{\max}\right)  \right\vert .
\end{equation}

\end{definition}

In the binary case $\varphi_{\left(  2,2\right)  }^{\left[  0,1\right]
}\left(  k_{\max}\right)  =2\varphi\left(  k_{\max}\right)  $. Because of
\emph{Remark \ref{rem-pol}}, the computation of the polyadic Euler function
requires for each element in the range $x_{-k_{\max}}<x_{k}<x_{k_{\max}}$ a
thorough consideration of composition sets.

\begin{example}
\textbf{1)} For the limiting polyadic ring $\mathbb{Z}_{\left(  30,2\right)
}^{\left[  1,29\right]  }$ and $k_{\max}=10$ we have $x_{-10}=-289$,
$x_{10}=291$ and%
\begin{align}
\mathbb{S}_{\left(  30,2\right)  }^{\left[  1,29\right]  }\left(  10\right)
&  =\left\{  -260,-202,-173,-115,-86,-28,1,59,88,146,175,233,262\right\}  ,\\
\varphi_{\left(  30,2\right)  }^{\left[  1,29\right]  }\left(  10\right)   &
=13.
\end{align}

\textbf{2)} In another limiting case with ternary multiplication
$\mathbb{Z}_{\left(  33,3\right)  }^{\left[  31,32\right]  }$ we get
$x_{-5}=-129$, $x_{5}=191$ and%
\begin{align}
\mathbb{S}_{\left(  33,3\right)  }^{\left[  31,32\right]  }\left(  5\right)
&  =\left\{  -97,-65,-1,31,95,127\right\}  ,\\
\varphi_{\left(  30,2\right)  }^{\left[  1,29\right]  }\left(  5\right)   &
=6.
\end{align}

\textbf{3)} In the non-limiting case $\mathbb{Z}_{\left(  11,5\right)
}^{\left[  7,10\right]  }$ and $k_{\max}=10$ we have $x_{-10}=-93$,
$x_{10}=107$ with%
\begin{align}
\mathbb{S}_{\left(  11,5\right)  }^{\left[  7,10\right]  }\left(  10\right)
&  =\left\{  -83,-73,-53,-43,-23,-13,7,17,37,47,67,77,97\right\}  ,\\
\varphi_{\left(  11,5\right)  }^{\left[  7,10\right]  }\left(  10\right)   &
=13.
\end{align}

\textbf{4)} For the polyadic Euler function in some other non-limiting cases
we have%
\begin{equation}
\varphi_{\left(  50,15\right)  }^{\left[  27,49\right]  }\left(  7\right)
=6,\ \ \ \varphi_{\left(  39,10\right)  }^{\left[  17,38\right]  }\left(
20\right)  =21,\ \ \ \varphi_{\left(  8,4\right)  }^{\left[  16,28\right]
}\left(  30\right)  =\varphi_{\left(  26,6\right)  }^{\left[  46,50\right]
}\left(  15\right)  =0.
\end{equation}

\end{example}

\subsection{The parameters-to-arity mapping}

Let us consider the connection between congruence classes and arities in more detail.

\begin{remark}
a) Solutions to (\ref{maa}) and (\ref{ana}) do not exist simultaneously
\textsl{for all} $a$ and $b$; b) The pair $a$, $b$ determines $m$, $n$
\textsl{uniquely}; c) It can occur that for several different pairs $a$, $b$
there can be \textsl{the same} arities $m$, $n$.
\end{remark}

Therefore, we have

\begin{assertion}
\label{asser-abmn}The \textit{parameters-to-arity} \textit{mapping}%
\begin{equation}
\psi:\left(  a,b\right)  \longrightarrow\left(  m,n\right)  \label{psi}%
\end{equation}
is a partial surjection.
\end{assertion}

Here we list the lowest arities which can be obtained with different choices
of $\left(  a,b\right)  $.{\tiny
\begin{align*}
\left.
\begin{array}
[c]{c}%
m=3\\
n=2
\end{array}
\right\}   &  :\left(
\begin{array}
[c]{c}%
a=1\\
b=2
\end{array}
\right)  ,\left(
\begin{array}
[c]{c}%
a=3\\
b=6
\end{array}
\right)  ,\left(
\begin{array}
[c]{c}%
a=5\\
b=10
\end{array}
\right)  ,\left(
\begin{array}
[c]{c}%
a=7\\
b=14
\end{array}
\right)  ,\left(
\begin{array}
[c]{c}%
a=9\\
b=18
\end{array}
\right)  ,\left(
\begin{array}
[c]{c}%
a=11\\
b=22
\end{array}
\right)  ,\left(
\begin{array}
[c]{c}%
a=13\\
b=26
\end{array}
\right)  ,\left(
\begin{array}
[c]{c}%
a=15\\
b=30
\end{array}
\right)  ;\\
\left.
\begin{array}
[c]{c}%
m=4\\
n=2
\end{array}
\right\}   &  :\left(
\begin{array}
[c]{c}%
a=1\\
b=3
\end{array}
\right)  ,\left(
\begin{array}
[c]{c}%
a=4\\
b=6
\end{array}
\right)  ,\left(
\begin{array}
[c]{c}%
a=4\\
b=12
\end{array}
\right)  ,\left(
\begin{array}
[c]{c}%
a=10\\
b=15
\end{array}
\right)  ,\left(
\begin{array}
[c]{c}%
a=7\\
b=21
\end{array}
\right)  ,\left(
\begin{array}
[c]{c}%
a=16\\
b=24
\end{array}
\right)  ,\left(
\begin{array}
[c]{c}%
a=10\\
b=30
\end{array}
\right)  ,\left(
\begin{array}
[c]{c}%
a=22\\
b=33
\end{array}
\right)  ;\\
\left.
\begin{array}
[c]{c}%
m=4\\
n=3
\end{array}
\right\}   &  :\left(
\begin{array}
[c]{c}%
a=2\\
b=3
\end{array}
\right)  ,\left(
\begin{array}
[c]{c}%
a=2\\
b=6
\end{array}
\right)  ,\left(
\begin{array}
[c]{c}%
a=8\\
b=12
\end{array}
\right)  ,\left(
\begin{array}
[c]{c}%
a=14\\
b=21
\end{array}
\right)  ,\left(
\begin{array}
[c]{c}%
a=8\\
b=24
\end{array}
\right)  ,\left(
\begin{array}
[c]{c}%
a=20\\
b=30
\end{array}
\right)  ,\left(
\begin{array}
[c]{c}%
a=11\\
b=33
\end{array}
\right)  ,\left(
\begin{array}
[c]{c}%
a=14\\
b=42
\end{array}
\right)  ;\\
\left.
\begin{array}
[c]{c}%
m=5\\
n=2
\end{array}
\right\}   &  :\left(
\begin{array}
[c]{c}%
a=1\\
b=4
\end{array}
\right)  ,\left(
\begin{array}
[c]{c}%
a=9\\
b=12
\end{array}
\right)  ,\left(
\begin{array}
[c]{c}%
a=5\\
b=20
\end{array}
\right)  ,\left(
\begin{array}
[c]{c}%
a=21\\
b=28
\end{array}
\right)  ,\left(
\begin{array}
[c]{c}%
a=9\\
b=36
\end{array}
\right)  ,\left(
\begin{array}
[c]{c}%
a=33\\
b=44
\end{array}
\right)  ,\left(
\begin{array}
[c]{c}%
a=13\\
b=52
\end{array}
\right)  ,\left(
\begin{array}
[c]{c}%
a=45\\
b=60
\end{array}
\right)  ;\\
\left.
\begin{array}
[c]{c}%
m=5\\
n=3
\end{array}
\right\}   &  :\left(
\begin{array}
[c]{c}%
a=3\\
b=4
\end{array}
\right)  ,\left(
\begin{array}
[c]{c}%
a=3\\
b=12
\end{array}
\right)  ,\left(
\begin{array}
[c]{c}%
a=15\\
b=20
\end{array}
\right)  ,\left(
\begin{array}
[c]{c}%
a=7\\
b=28
\end{array}
\right)  ,\left(
\begin{array}
[c]{c}%
a=27\\
b=36
\end{array}
\right)  ,\left(
\begin{array}
[c]{c}%
a=39\\
b=52
\end{array}
\right)  ,\left(
\begin{array}
[c]{c}%
a=15\\
b=60
\end{array}
\right)  ,\left(
\begin{array}
[c]{c}%
a=51\\
b=68
\end{array}
\right)  .
\end{align*}
}

Although it has not been possible to derive a general formula for $\psi\left(  a,b\right)  $, this can be done in
some particular cases

\begin{proposition}
\textbf{1)} In the limiting cases ($a=1,b-1$) we have%
\begin{equation}
\psi\left(  1,b\right)  =\left(  b+1,2\right)  ,\ \ \ \ \ \ \ \ \psi\left(
b-1,b\right)  =\left(  b+1,3\right)  .
\end{equation}

\textbf{2)} If $a\mid b$, then%
\begin{equation}
\psi\left(  a,ad\right)  =\left(  d+1,\min\limits_{l}\log_{a}\left(
ld+1\right)  +1\right)  ,
\end{equation}
where $l$ is the smallest integer for which $\log$ takes its minimal integer value.

\textbf{3)} If $\gcd\left(  a,b\right)  =d$, then%
\begin{equation}
\psi\left(  a,b\right)  =\psi\left(  a_{0}d,b_{0}d\right)  =\left(
b_{0}+1,\min\limits_{l}\log_{a}\left(  l\dfrac{b_{0}}{a_{0}}+1\right)
+1\right)  ,
\end{equation}
with the same $l$.
\end{proposition}

\begin{proof}
All the statements follow directly from (\ref{maa})--(\ref{ana}).
\end{proof}

In our approach, the concrete choice of operations (\ref{nu})--(\ref{mu})
inside a conguence class $\left[  \left[  a\right]  \right]  _{b}$ gives

\begin{assertion}
\label{asser-mn}The number of additions in the polyadic ring $\mathbb{Z}%
_{\left(  m,n\right)  }^{\left[  a,b\right]  }$ is greater that the number of
mutiplications $m>n$ with any choice of $\left(  a,b\right)  $.
\end{assertion}

Also, not all pairs $\left(  a,b\right)  $ are allowed due to (\ref{maa}%
)--(\ref{ana}). We list the \textsl{forbidden} $\left(  a,b\right)  $ for
$b\leq20$%
\begin{align}
b  &  =4:a=2;\nonumber\\
b  &  =8:a=2,4,6;\nonumber\\
b  &  =9:a=3,6;\nonumber\\
b  &  =12:a=2,6,10;\nonumber\\
b  &  =16:a=2,4,6,8,12,14;\nonumber\\
b  &  =18:a=3,6,12,15;\nonumber\\
b  &  =20:a=2,6,10,14,18. \label{b}%
\end{align}

The characterization of the fixed congruence class $\left[  \left[  a\right]
\right]  _{b}$ and the corresponding $\left(  m,n\right)  $-ring of polyadic
integer numbers $\mathbb{Z}_{\left(  m,n\right)  }^{\left[  a,b\right]  }$ can
be done in terms of the \textit{shape invariants} $I,J\in\mathbb{Z}_{+}$
defined uniquely by (\textsc{Table 3} in \cite{dup2017})%
\begin{equation}
I=I_{m}^{\left[  a,b\right]  }=\left(  m-1\right)  \dfrac{a}{b}%
,\ \ \ \ \ \ \ J=J_{n}^{\left[  a,b\right]  }=\dfrac{a^{n}-a}{b}. \label{ij}%
\end{equation}

Obviously, in the binary case, when $m=n=2$ ($a=0,b=1$) both shape invariants
vanish, $I=J=0$. Nevertheless, there exist \textquotedblleft
partially\textquotedblright\ binary cases, when only $n=2$ and $m\neq2$, while
$J$ is nonzero, for instance in $\mathbb{Z}_{\left(  6,2\right)  }^{\left[
6,10\right]  }$ we have $I=J=3$. In Example \ref{exam-34} for $\mathbb{Z}%
_{\left(  5,3\right)  }^{\left[  3,4\right]  }$ we have $I=3,\ J=6$. In the
limiting cases (\ref{xk1})--(\ref{xk2}) we have, in general, for a fixed $b$%
\begin{align}
I_{b+1}^{\left[  1,b\right]  }  &  =1,\ \ \ \ J_{2}^{\left[  1,b\right]
}=0,\label{ij1}\\
I_{b+1}^{\left[  b-1,b\right]  }  &  =b-1,\ \ \ \ J_{3}^{\left[  b-1,b\right]
}=\left(  b-1\right)  \left(  b-2\right)  . \label{ij2}%
\end{align}
Thus, one can classify and distinguish the limiting cases of the congruence
classes in terms of the invariants and their manifest form (\ref{ij1}%
)--(\ref{ij2}).

\section{\textsc{Finite polyadic rings}}

Now we present a special method of constructing a finite \textsl{nonderived}
polyadic ring by combining the \textquotedblleft external\textquotedblright%
\ and \textquotedblleft internal\textquotedblright\ methods. Let us
\textquotedblleft apply\textquotedblright\ \textbf{1)} to \textbf{2)}, such
that instead of (\ref{zb}), we introduce the finite polyadic ring
$\mathbb{Z}_{\left(  m,n\right)  }\diagup c\mathbb{Z}$, where $\mathbb{Z}%
_{\left(  m,n\right)  }$ is defined in (\ref{za}). However, if we directly
consider the \textquotedblleft double\textquotedblright\ class $\left\{
a+bk+cl\right\}  $ and fix $a$ and $b$, then the factorization by
$c\mathbb{Z}$ will not give closed operations for arbitrary $c$.

\begin{assertion}
If the finite polyadic ring $\mathbb{Z}_{\left(  m,n\right)  }^{\left[
a,b\right]  }\diagup c\mathbb{Z}$ has $q$ elements, then%
\begin{equation}
c=bq. \label{bq}%
\end{equation}

\end{assertion}

\begin{proof}
It follows from (\ref{bq}), that the \textquotedblleft
double\textquotedblright\ class remains in $\left[  \left[  a\right]  \right]
_{b}$.
\end{proof}

\begin{remark}
The representatives $x_{k}^{\left[  a,b\right]  }=a+bk$ belong to a $\left(
m,n\right)  $-ring with the polyadic operations (\ref{nu})--(\ref{mu}), while
the notions of subtraction, division, modulo and remainder are defined for
binary operations. Therefore, we cannot apply the standard binary modular
arithmetic to $x_{k}^{\left[  a,b\right]  }$ directly, but we can define the
equivalence relations and corresponding class operations in terms of $k$.
\end{remark}

\subsection{Secondary congruence classes}

On the set of the \textquotedblleft double\textquotedblright\ classes
$\left\{  a+bk+bql\right\}  $, $k,l\in\mathbb{Z}$ and fixed $b\in\mathbb{N}$
and $a=0,\ldots,b-1$ we define the equivalence relation $\overset{k}{\sim}$ by%
\begin{equation}
\left\{  a+bk_{1}+bql_{1}\right\}  \overset{k}{\sim}\left\{  a+bk_{2}%
+bql_{2}\right\}  \Longrightarrow k_{1}-k_{2}=ql,\ \ \ l,l_{1},l_{2}%
\in\mathbb{Z}. \label{kk}%
\end{equation}

\begin{proposition}
The equivalence relation $\overset{k}{\sim}$ is a congruence.
\end{proposition}

\begin{proof}
It follows from the \textsl{binary} congruency of $k$'s such that $k_{1}\equiv
k_{2}\left(  \operatorname{mod}q\right)  $ which is just a rewritten form of
the last condition of (\ref{kk}).
\end{proof}

So now we can factorize $\mathbb{Z}_{\left(  m,n\right)  }^{\left[
a,b\right]  }$ by the congruence $\overset{k}{\sim}$ and obtain

\begin{definition}
A \textit{secondary (equivalence) class} of a polyadic integer $x_{k}^{\left[
a,b\right]  }=a+bk\in\mathbb{Z}_{\left(  m,n\right)  }^{\left[  a,b\right]  }$
\textquotedblleft modulo\textquotedblright\ $bq$ (with $q$ being the number of
representatives $x_{k}^{\left[  a,b\right]  }$, for fixed $b\in\mathbb{N}$ and
$0\leq a\leq b-1$) is%
\begin{equation}
\left[  \left[  x_{k}^{\left[  a,b\right]  }\right]  \right]  _{bq}=\left\{
\left\{  \left(  a+bk\right)  +bql\right\}  \mid l\in\mathbb{Z},\ q\in
\mathbb{N},\ 0\leq k\leq q-1\right\}  . \label{abq}%
\end{equation}

\end{definition}

\begin{remark}
If the binary limit is given by $a=0,b=1$ and $\mathbb{Z}_{\left(  2,2\right)
}^{\left[  0,1\right]  }=\mathbb{Z}$, then the secondary class becomes the
ordinary class (\ref{ab}).
\end{remark}

If the values of the parameters $a,b,q$ are clear from the context, we denote
the secondary class representatives by an integer with \textsl{two} primes, as
follows $\left[  \left[  x_{k}^{\left[  a,b\right]  }\right]  \right]
_{bq}\equiv x_{k}^{\prime\prime}\equiv x^{\prime\prime}$.

\begin{example}
\textbf{a)} Let $a=3,b=5$, then for $q=4$ elements we have the secondary
classes with $k=0,1,2,3$ (the corresponding binary limits are in brackets)%
\begin{align}
\left[  \left[  x_{k}^{\left[  3,5\right]  }\right]  \right]  _{20}  &
=3^{\prime\prime},8^{\prime\prime},13^{\prime\prime},18^{\prime\prime
}=\left\{
\begin{array}
[c]{c}%
3^{\prime\prime}=\left\{  \ldots-17,3,23,43,63,\ldots\right\}  ,\\
8^{\prime\prime}=\left\{  \ldots-12,8,28,48,68,\ldots\right\}  ,\\
13^{\prime\prime}=\left\{  \ldots-7,13,33,53,73,\ldots\right\}  ,\\
18^{\prime\prime}=\left\{  \ldots-2,18,38,58,78,\ldots\right\}  ,
\end{array}
\right. \label{x35}\\
&  \left(  \left[  \left[  k\right]  \right]  _{4}=0^{\prime},1^{\prime
},2^{\prime},3^{\prime}=\left\{
\begin{array}
[c]{c}%
0^{\prime}=\left\{  \ldots-4,0,4,8,12,\ldots\right\}  ,\\
1^{\prime}=\left\{  \ldots-3,1,5,9,13,\ldots\right\}  ,\\
2^{\prime}=\left\{  \ldots-2,2,6,10,14,\ldots\right\}  ,\\
3^{\prime}=\left\{  \ldots-1,3,7,11,15,\ldots\right\}  .
\end{array}
\right.  \right)
\end{align}

\textbf{b)} For $a=3,b=6$ and for $4$ elements and $k=0,1,2,3$%
\begin{equation}
\left[  \left[  x_{k}^{\left[  3,6\right]  }\right]  \right]  _{24}%
=3^{\prime\prime},9^{\prime\prime},15^{\prime\prime},21^{\prime\prime
},\ \ \ \ \ \ \ \ \ \ \ \ \ \ \left(  \left[  \left[  k\right]  \right]
_{4}=0^{\prime},1^{\prime},2^{\prime},3^{\prime}\right)  . \label{x36}%
\end{equation}

\textbf{c)} If $a=4,b=5$, for $3$ elements and $k=0,1,2$ we get%
\begin{equation}
\left[  \left[  x_{k}^{\left[  4,5\right]  }\right]  \right]  _{15}%
=4^{\prime\prime},9^{\prime\prime},14^{\prime\prime}%
,\ \ \ \ \ \ \ \ \ \ \ \ \ \ \left(  \left[  \left[  k\right]  \right]
_{3}=0^{\prime},1^{\prime},2^{\prime}\right)  . \label{x45}%
\end{equation}

The crucial difference between these sets of classes are: 1) they are
described by rings of different arities determined by (\ref{maa}) and
(\ref{ana}); 2) some of them are fields.
\end{example}

\subsection{Finite polyadic rings of secondary classes}

Now we determine the operations between secondary classes. The most
significant difference with the binary class operations (\ref{a1})--(\ref{a2})
is the fact that secondary classes obey \textsl{nonderived} \textsl{polyadic} operations.

\begin{proposition}
\label{prop-xk}The set $\left\{  x_{k}^{\prime\prime}\right\}  $ of $q$
secondary classes $k=0,\ldots,q-1$ (with the fixed $a,b$) can be endowed with
the following commutative $m$-ary addition%
\begin{align}
x_{k_{add}}^{\prime\prime}  &  =\nu_{m}^{\prime\prime}\left[  x_{k_{1}%
}^{\prime\prime},x_{k_{1}}^{\prime\prime},\ldots,x_{k_{m}}^{\prime\prime
}\right]  ,\label{xadd}\\
k_{add}  &  \equiv\left(  \left(  k_{1}+k_{2}+\ldots+k_{m}\right)
+I_{m}^{\left[  a,b\right]  }\right)  \left(  \operatorname{mod}q\right)
\label{kadd}%
\end{align}
and commutative $n$-ary multiplication%
\begin{align}
x_{k_{mult}}^{\prime\prime}  &  =\mu_{n}^{\prime\prime}\left[  x_{k_{1}%
}^{\prime\prime},x_{k_{1}}^{\prime\prime},\ldots,x_{k_{n}}^{\prime\prime
}\right]  ,\label{xmult}\\
k_{mult}  &  \equiv\left(  a^{n-1}\left(  k_{1}+k_{2}+\ldots+k_{n}\right)
+a^{n-2}b\left(  k_{1}k_{2}+k_{2}k_{3}+\ldots+k_{n-1}k_{n}\right)
+\ldots\right. \nonumber\\
&  \left.  +b^{n-1}k_{1}\ldots k_{n}+J_{n}^{\left[  a,b\right]  }\right)
\left(  \operatorname{mod}q\right)  , \label{kmult}%
\end{align}
which satisfy the polyadic distributivity, and the shape invariants
$I_{m}^{\left[  a,b\right]  },J_{n}^{\left[  a,b\right]  }$ are defined in
(\ref{ij}).
\end{proposition}

\begin{proof}
This follows from the definition of the secondary class (\ref{abq}) and
manifest form of the \textquotedblleft underlying\textquotedblright\ polyadic
operations (\ref{nu})--(\ref{mu}), which are commutative and distributive.
\end{proof}

\begin{remark}
The binary limit is given by $a=0,b=1$ and $m=n=2$, $I_{m}^{\left[
a,b\right]  }=J_{n}^{\left[  a,b\right]  }=0$, such that the secondary class
becomes the ordinary congruence class $x_{k}^{\prime\prime}\rightarrow
k^{\prime}$, obeying the standard binary class operations (\ref{a1}%
)--(\ref{a2}), which in terms of $k$ are $k_{add}\equiv\left(  k_{1}%
+k_{2}\right)  \left(  \operatorname{mod}q\right)  $, $k_{mult}\equiv\left(
k_{1}k_{2}\right)  \left(  \operatorname{mod}q\right)  $.
\end{remark}

\begin{definition}
The set of secondary classes (\ref{abq}) equipped with operations
(\ref{xadd}), (\ref{xmult}) is denoted by%
\begin{equation}
\mathbb{Z}_{\left(  m,n\right)  }\left(  q\right)  \equiv\mathbb{Z}_{\left(
m,n\right)  }^{\left[  a,b\right]  }\left(  q\right)  =\mathbb{Z}_{\left(
m,n\right)  }^{\left[  a,b\right]  }\diagup\left(  bq\right)  \mathbb{Z=}%
\left\{  \left\{  x_{k}^{\prime\prime}\right\}  \mid\nu_{m}^{\prime\prime}%
,\mu_{n}^{\prime\prime}\right\}  , \label{zmn}%
\end{equation}
and is a \textit{finite secondary class} $\left(  m,n\right)  $-\textit{ring}
of polyadic integer numbers $\mathbb{Z}_{\left(  m,n\right)  }\equiv
\mathbb{Z}_{\left(  m,n\right)  }^{\left[  a,b\right]  }$. The value $q$ (the
number of elements) is called its \textit{order}.
\end{definition}

Informally, $\mathbb{Z}_{\left(  m,n\right)  }=\mathbb{Z}_{\left(  m,n\right)
}\left(  \infty\right)  $. First, note that the constructed finite $\left(
m,n\right)  $-rings (\ref{zmn}) have a much richer structure and exotic
properties which do not exist in the binary finite residue class rings
(\ref{zb}), and, in general, they give many concrete examples for possible
different kinds polyadic rings. One of such \textquotedblleft
non-binary\textquotedblright\ properties is the availability of several units
(for odd multiplicative arity $n$), and moreover sometimes all ring elements
are units (such rings are \textquotedblleft automatically\textquotedblright%
\ fields, see below).

\begin{example}
\textbf{a)} In $\left(  5,3\right)  $-ring $\mathbb{Z}_{\left(  4,3\right)
}^{\left[  3,4\right]  }\left(  2\right)  $ with 2 secondary classes both
elements are units (we mark units by subscript $e$) $e_{1}=3_{e}^{\prime
\prime}=3^{\prime\prime}$, $e_{2}=7_{e}^{\prime\prime}=7^{\prime\prime}$,
because they are both multiplicative idempotents and satisfy the following
ternary multiplication (cf. (\ref{e}))%
\begin{equation}
\mu_{3}\left[  3^{\prime\prime},3^{\prime\prime},3^{\prime\prime}\right]
=3^{\prime\prime},\ \mu_{3}\left[  3^{\prime\prime},3^{\prime\prime}%
,7^{\prime\prime}\right]  =7^{\prime\prime},\ \mu_{3}\left[  3^{\prime\prime
},7^{\prime\prime},7^{\prime\prime}\right]  =3^{\prime\prime},\ \mu_{3}\left[
7^{\prime\prime},7^{\prime\prime},7^{\prime\prime}\right]  =7^{\prime\prime}.
\end{equation}

\textbf{b)} In the same way the ring $\mathbb{Z}_{\left(  7,3\right)
}^{\left[  5,6\right]  }\left(  4\right)  $ consists of only 4 units
$e_{1}=5_{e}^{\prime\prime},e_{2}=11_{e}^{\prime\prime},e_{3}=17_{e}%
^{\prime\prime},e_{4}=23_{e}^{\prime\prime}$, and no zero.

\textbf{c)} Equal arity rings of the same order may be not isomorphic. For
instance, $\mathbb{Z}_{\left(  4,2\right)  }^{\left[  1,3\right]  }\left(
2\right)  $ consists of unit $e=1_{e}^{\prime\prime}=1^{\prime\prime}$ and
zero $z=4_{z}^{\prime\prime}=4^{\prime\prime}$ only, satisfying%
\begin{equation}
\mu_{2}\left[  1^{\prime\prime},1^{\prime\prime}\right]  =1^{\prime\prime
},\ \mu_{2}\left[  1^{\prime\prime},4^{\prime\prime}\right]  =4^{\prime\prime
},\ \mu_{2}\left[  4^{\prime\prime},4^{\prime\prime}\right]  =4^{\prime\prime
},
\end{equation}
and therefore $\mathbb{Z}_{\left(  4,2\right)  }^{\left[  1,3\right]  }\left(
2\right)  $ is a field, because $\left\{  1^{\prime\prime},4_{z}^{\prime
\prime}\right\}  \setminus4_{z}^{\prime\prime}$ is a (trivial) binary group,
consisting of one element $1_{e}^{\prime\prime}$. However, $\mathbb{Z}%
_{\left(  4,2\right)  }^{\left[  4,6\right]  }\left(  2\right)  $ has the zero
$z=4_{z}^{\prime\prime}=4^{\prime\prime}$, $10^{\prime\prime}$ and has no
unit, because%
\begin{equation}
\mu_{2}\left[  4^{\prime\prime},4^{\prime\prime}\right]  =4^{\prime\prime
},\ \mu_{2}\left[  4^{\prime\prime},10^{\prime\prime}\right]  =4^{\prime
\prime},\ \mu_{2}\left[  10^{\prime\prime},10^{\prime\prime}\right]
=4^{\prime\prime},
\end{equation}
so that $\mathbb{Z}_{\left(  4,2\right)  }^{\left[  4,6\right]  }\left(
2\right)  $ is not a field, because of the last relation (nilpotency of
$10^{\prime\prime}$). Their additive $4$-ary groups are also not isomorphic
(which is easy to show). However, $\mathbb{Z}_{\left(  4,2\right)  }^{\left[
1,3\right]  }\left(  2\right)  $ and $\mathbb{Z}_{\left(  4,2\right)
}^{\left[  4,6\right]  }\left(  2\right)  $ have the same arity and order.
\end{example}

Recalling \textbf{Assertion \ref{asser-abmn}}, we conclude more concretely:

\begin{assertion}
For a fixed arity shape $\left(  m,n\right)  $, there can be non-isomorphic
secondary class polyadic rings $\mathbb{Z}_{\left(  m,n\right)  }\left(
q\right)  $ of the same order $q$, which describe different binary residue
classes $\left[  \left[  a\right]  \right]  _{b}$.
\end{assertion}

A polyadic analog of the characteristic can be introduced, when there exist
both a unit and zero in a finite ring. Recall, that if $\mathcal{R}$ is a
finite binary ring with unit$1$ and zero $0$, then its characteristic is
defined as a smallest integer $\chi$, such that
\begin{equation}
\left(  \overset{\chi}{\overbrace{1+1+,\ldots,+1}}\right)  =\chi\cdot1=0.
\label{11}%
\end{equation}
This means that the \textquotedblleft number\textquotedblright\ of unit
additions being $\chi-1$ produces zero. The same is evident for any other
element $x\in\mathcal{R}$, because $x=x\cdot1$.

\begin{definition}
For the finite polyadic ring $\mathbb{Z}_{\left(  m,n\right)  }\left(
q\right)  $ which contains both the unit $e$ and the zero $z$, a
\textit{polyadic characteristic} $\chi_{p}$ is defined as a smallest additive
polyadic power (\ref{xx}) of $e$ which is equal to zero%
\begin{equation}
e^{\left\langle \chi_{p}\right\rangle _{+m}}=z. \label{ez}%
\end{equation}

\end{definition}

In the binary limit, obviously, $\chi_{p}=\chi-1$. A polyadic analog of the
middle term in (\ref{11}) can be obtained by using the polyadic distributivity
and (\ref{nu}) as%
\begin{equation}
e^{\left\langle \chi_{p}\right\rangle _{+m}}=e^{\left\langle \chi_{p}\left(
m-1\right)  +1\right\rangle _{\times n}}. \label{ee}%
\end{equation}

In \textsc{Table \ref{T0}} we present the parameters-to-arity mapping\textit{
}$\psi_{\left(  m,n\right)  }^{\left[  a,b\right]  }$ (\ref{psi}) together
with the polyadic characteristics of those finite secondary class rings
$\mathbb{Z}_{\left(  m,n\right)  }^{\left[  a,b\right]  }\left(  q\right)  $
which contain both unit(s) and zero, and which have order less or equal than
$10$ for $b\leq6$.

\begin{table}[pb]
\caption[Polyadic characteristics of finite $\left(  m,n\right)  $%
-rings]{Polyadic characteristics $\chi_{p}$ for the finite secondary class
$\left(  m,n\right)  $-rings $\mathbb{Z}_{m,n}^{\left[  a,b\right]  }\left(
q\right)  $ of order $2\leq q\leq10$ for $2\leq b\leq6$. The orders $q$ which
do not give fields are \textsl{slanted}.}%
\label{T0}
\begin{center}
\begin{tabular}
[c]{||c||c|c|c|c|c|}\hline\hline
$a\setminus b$ & 2 & 3 & 4 & 5 & 6\\\hline\hline
1 & $%
\begin{array}
[c]{c}%
m=\mathbf{3}\\
n=\mathbf{2}\\
\text{\textsf{q=3}},\chi_{p}=1\\
\text{\textsf{q=5}},\chi_{p}=2\\
\text{\textsf{q=7}},\chi_{p}=3\\
\text{\textsl{q=9}},\chi_{p}=4
\end{array}
$ & $%
\begin{array}
[c]{c}%
m=\mathbf{4}\\
n=\mathbf{2}\\
\text{\textsf{q=2}},\chi_{p}=1\\
\text{\textsl{q=4}},\chi_{p}=1\\
\text{\textsf{q=5}},\chi_{p}=3\\
\text{\textsf{q=7}},\chi_{p}=2\\
\text{\textsl{q=8}},\chi_{p}=5\\
\text{\textsl{q=10}},\chi_{p}=3
\end{array}
$ & $%
\begin{array}
[c]{c}%
m=\mathbf{5}\\
n=\mathbf{2}\\
\text{\textsf{q=3}},\chi_{p}=2\\
\text{\textsf{q=5}},\chi_{p}=1\\
\text{\textsf{q=7}},\chi_{p}=5\\
\text{\textsl{q=9}},\chi_{p}=2
\end{array}
$ & $%
\begin{array}
[c]{c}%
m=\mathbf{6}\\
n=\mathbf{2}\\
\text{\textsf{q=2}},\chi_{p}=1\\
\text{\textsf{q=3}},\chi_{p}=1\\
\text{\textsl{q=4}},\chi_{p}=3\\
\text{\textsl{q=6}},\chi_{p}=1\\
\text{\textsf{q=7}},\chi_{p}=4\\
\text{\textsl{q=8}},\chi_{p}=3\\
\text{\textsl{q=9}},\chi_{p}=7
\end{array}
$ & $%
\begin{array}
[c]{c}%
m=\mathbf{7}\\
n=\mathbf{2}\\
\text{\textsf{q=5}},\chi_{p}=4\\
\text{\textsf{q=7}},\chi_{p}=1
\end{array}
$\\\hline
2 &  & $%
\begin{array}
[c]{c}%
m=\mathbf{4}\\
n=\mathbf{3}\\
\text{\textsf{q=2}},\chi_{p}=1\\
\text{\textsl{q=4}},\chi_{p}=1\\
\text{\textsf{q=5}},\chi_{p}=3\\
\text{\textsf{q=7}},\chi_{p}=2\\
\text{\textsl{q=8}},\chi_{p}=5\\
\text{\textsl{q=10}},\chi_{p}=3
\end{array}
$ &  & $%
\begin{array}
[c]{c}%
m=\mathbf{6}\\
n=\mathbf{5}\\
\text{\textsf{q=2}},\chi_{p}=1\\
\text{\textsf{q=3}},\chi_{p}=1\\
\text{\textsl{q=4}},\chi_{p}=3\\
\text{\textsl{q=6}},\chi_{p}=1\\
\text{\textsf{q=7}},\chi_{p}=4\\
\text{\textsl{q=8}},\chi_{p}=3\\
\text{\textsl{q=9}},\chi_{p}=7
\end{array}
$ & $%
\begin{array}
[c]{c}%
m=\mathbf{4}\\
n=\mathbf{3}\\
\text{\textsf{q=5}},\chi_{p}=3\\
\text{\textsf{q=7}},\chi_{p}=2\\
\text{\textsl{q=10}},\chi_{p}=3
\end{array}
$\\\hline
3 &  &  & $%
\begin{array}
[c]{c}%
m=\mathbf{5}\\
n=\mathbf{3}\\
\text{\textsf{q=3}},\chi_{p}=2\\
\text{\textsf{q=5}},\chi_{p}=1\\
\text{\textsf{q=7}},\chi_{p}=5\\
\text{\textsl{q=9}},\chi_{p}=2
\end{array}
$ & $%
\begin{array}
[c]{c}%
m=\mathbf{6}\\
n=\mathbf{5}\\
\text{\textsf{q=2}},\chi_{p}=1\\
\text{\textsf{q=3}},\chi_{p}=1\\
\text{\textsl{q=4}},\chi_{p}=3\\
\text{\textsl{q=6}},\chi_{p}=1\\
\text{\textsf{q=7}},\chi_{p}=4\\
\text{\textsl{q=8}},\chi_{p}=3
\end{array}
$ & $%
\begin{array}
[c]{c}%
m=\mathbf{3}\\
n=\mathbf{2}\\
\text{\textsf{q=5}},\chi_{p}=2\\
\text{\textsf{q=7}},\chi_{p}=3
\end{array}
$\\\hline
4 &  &  &  & $%
\begin{array}
[c]{c}%
m=\mathbf{6}\\
n=\mathbf{3}\\
\text{\textsf{q=2}},\chi_{p}=1\\
\text{\textsf{q=3}},\chi_{p}=1\\
\text{\textsl{q=4}},\chi_{p}=3\\
\text{\textsl{q=6}},\chi_{p}=1\\
\text{\textsf{q=7}},\chi_{p}=4\\
\text{\textsl{q=8}},\chi_{p}=3\\
\text{\textsl{q=9}},\chi_{p}=7
\end{array}
$ & $%
\begin{array}
[c]{c}%
m=\mathbf{4}\\
n=\mathbf{2}\\
\text{\textsf{q=5}},\chi_{p}=3\\
\text{\textsf{q=7}},\chi_{p}=2\\
\text{\textsl{q=10}},\chi_{p}=3
\end{array}
$\\\hline
5 &  &  &  &  & $%
\begin{array}
[c]{c}%
m=\mathbf{7}\\
n=\mathbf{3}\\
\text{\textsf{q=5}},\chi_{p}=4\\
\text{\textsf{q=7}},\chi_{p}=1
\end{array}
$\\\hline\hline
\end{tabular}
\end{center}
\end{table}

Now we turn to the question of which secondary classes can be described by
polyadic finite fields.

\section{\textsc{Finite polyadic fields}}

Let us consider the structure of the finite secondary class rings
$\mathbb{Z}_{\left(  m,n\right)  }^{\left[  a,b\right]  }\left(  q\right)  $
in more detail and determine which of them are polyadic fields.

\begin{proposition}
A finite polyadic ring $\mathbb{Z}_{\left(  m,n\right)  }^{\left[  a,b\right]
}\left(  q\right)  $ is a \textit{secondary class finite }$\left(  m,n\right)
$-\textit{field} $\mathbb{F}_{\left(  m,n\right)  }^{\prime\prime\left[
a,b\right]  }\left(  q\right)  $ if all its elements except $z$ (if it exists)
are polyadically multiplicative invertible having a unique querelement.
\end{proposition}

\begin{proof}
In both cases $\left\{  \left\{  x_{k}^{\prime\prime}\right\}  \mid\mu
_{n}^{\prime\prime}\right\}  $ and $\left\{  \left\{  x_{k}^{\prime\prime
}\setminus z\right\}  \mid\mu_{n}^{\prime\prime}\right\}  $ are commutative
and cancellative $n$-ary groups, which follows from the concrete form of
multiplication (\ref{xmult}). Therefore, according to \textbf{Definition
\ref{def-fmn}}, in such a case $\mathbb{Z}_{\left(  m,n\right)  }^{\left[
a,b\right]  }\left(  q\right)  $ becomes a polyadic field $\mathbb{F}_{\left(
m,n\right)  }^{\prime\prime\left[  a,b\right]  }\left(  q\right)  $.
\end{proof}

\subsection{Abstract finite polyadic fields}

In the binary case \cite{lid/nie} the residue (congruence) class ring
(\ref{zb}) with $q$ elements $\mathbb{Z}\diagup q\mathbb{Z}$ is a congruence
class (non-extended) field, if its order $q=p$ is a prime number, such that
$\mathbb{F}^{\prime}\left(  p\right)  =\left\{  \left\{  \left[  \left[
a\right]  \right]  _{p}\right\}  \mid+^{\prime},\cdot^{\prime},0^{\prime
},1^{\prime}\right\}  $, $a=0,1,\ldots,p-1$. Because all non-extended binary
fields of a fixed prime order $p$ are isomorphic each other and, in tern,
isomorphic to the congruence class field $\mathbb{F}^{\prime}\left(  p\right)
$, it is natural to study them in a more \textquotedblleft
abstract\textquotedblright\ way, i.e. without connection to a specific
congruence class structure. This can be achieved by consideration of the
one-to-one onto mapping from the congruence class to its representative which
preserves the field (ring) structure and provides operations (binary
multiplication and addition with ordinary $0$ and $1$) by modulo $p$. In other
words, the mapping $\Phi_{p}\left(  \left[  \left[  a\right]  \right]
_{p}\right)  =a$ is an isomorphism of binary fields $\Phi_{p}:\mathbb{F}%
^{\prime}\left(  p\right)  \rightarrow\mathbb{F}\left(  p\right)  $, where
$\mathbb{F}\left(  p\right)  =\left\{  \left\{  a\right\}  \mid+,\cdot
,0,1\right\}  _{\operatorname{mod}p}$ is an \textquotedblleft
abstract\textquotedblright\ non-extended (prime) finite field of order $p$ (or
\textit{Galois field} $GF\left(  p\right)  $).

In a similar way, we introduce a polyadic analog of the \textquotedblleft
abstract\textquotedblright\ binary non-extended (prime) finite fields. Let us
consider the set of \textsl{polyadic integer numbers} $\left\{  x_{k}\right\}
\equiv\left\{  x_{k}^{\left[  a,b\right]  }\right\}  =\left\{  a+bk\right\}
\in\mathbb{Z}_{\left(  m,n\right)  }^{\left[  a,b\right]  }$, $b\in\mathbb{N}$
and $0\leq a\leq b-1$,$\ 0\leq k\leq q-1$, $q\in\mathbb{N}$, which obey the
operations (\ref{nu})--(\ref{mu}). The polyadic version of the prime finite
field $\mathbb{F}\left(  p\right)  $ of order $p$ (or Galois field $GF\left(
p\right)  $) is given by

\begin{definition}
The \textquotedblleft abstract\textquotedblright\ \textit{non-extended (prime)
finite} $\left(  m,n\right)  $-\textit{field of order }$q$ is
\begin{equation}
\mathbb{F}_{\left(  m,n\right)  }\left(  q\right)  \equiv\mathbb{F}_{\left(
m,n\right)  }^{\left[  a,b\right]  }\left(  q\right)  =\left\{  \left\{
a+bk\right\}  \mid\nu_{m},\mu_{n}\right\}  _{\operatorname{mod}bq},
\label{fmn}%
\end{equation}
if $\left\{  \left\{  x_{k}\right\}  \mid\nu_{m}\right\}  _{\operatorname{mod}%
bq}$ is an additive $m$-ary group, and $\left\{  \left\{  x_{k}\right\}
\mid\mu_{n}\right\}  _{\operatorname{mod}bq}$ (or, when zero $z$ exists,
$\left\{  \left\{  x_{k}\setminus z\right\}  \mid\mu_{n}\right\}
_{\operatorname{mod}bq}$) is a multiplicative $n$-ary group.
\end{definition}

Then we define a one-to-one onto mapping from the secondary congruence class
to its representative by $\Phi_{q}^{\left[  a,b\right]  }\left(  \left[
\left[  x_{k}^{\left[  a,b\right]  }\right]  \right]  _{bq}\right)
=x_{k}^{\left[  a,b\right]  }$ and arrive at the following

\begin{proposition}
The mapping $\Phi_{q}^{\left[  a,b\right]  }:\mathbb{F}_{\left(  m,n\right)
}^{\prime\prime\left[  a,b\right]  }\left(  q\right)  \rightarrow
\mathbb{F}_{\left(  m,n\right)  }^{\left[  a,b\right]  }\left(  q\right)  $ is
a polyadic ring homomorphism (being, in fact, an isomorphism) and satisfies
(here we use the \textquotedblleft prime\textquotedblright\ notations)%
\begin{align}
\Phi_{q}^{\left[  a,b\right]  }\left(  \nu_{m}^{\prime\prime}\left[
x_{1}^{\prime\prime},x_{2}^{\prime\prime},\ldots,x_{m}^{\prime\prime}\right]
\right)   &  =\nu_{m}\left[  \Phi_{q}^{\left[  a,b\right]  }\left(
x_{1}\right)  ,\Phi_{q}^{\left[  a,b\right]  }\left(  x_{2}\right)
,\ldots,\Phi_{q}^{\left[  a,b\right]  }\left(  x_{m}\right)  \right]  ,\\
\Phi_{q}^{\left[  a,b\right]  }\left(  \mu_{n}^{\prime\prime}\left[
x_{1}^{\prime\prime},x_{2}^{\prime\prime},\ldots,x_{n}^{\prime\prime}\right]
\right)   &  =\mu_{n}\left[  \Phi_{q}^{\left[  a,b\right]  }\left(
x_{1}\right)  ,\Phi_{q}^{\left[  a,b\right]  }\left(  x_{2}\right)
,\ldots,\Phi_{q}^{\left[  a,b\right]  }\left(  x_{n}\right)  \right]  .
\end{align}

\end{proposition}

\begin{proof}
This follows directly from (\ref{nu})--(\ref{mu}) and (\ref{xadd}%
)--(\ref{xmult}). Obviously, a mapping defined in this way governs the
polyadic distributivity, and therefore $\Phi_{q}^{\left[  a,b\right]  }$ is a
ring homomorphism, or, more exactly, a 1-place heteromorphism for $m$-ary
addition together with $n$-ary multiplication (see \cite{dup2017}). Because
$x_{k}^{\prime\prime}\rightarrow x_{k}$ is one-to-one for any fixed $0\leq
k\leq q-1$, $\Phi_{q}^{\left[  a,b\right]  }$ is an isomorphism.
\end{proof}

In \textsc{Table \ref{T1}} we present the \textquotedblleft
abstract\textquotedblright\ non-extended polyadic\textit{ finite }fields
$\mathbb{F}_{\left(  m,n\right)  }^{\left[  a,b\right]  }\left(  q\right)  $
of lowest arity shape $\left(  m,n\right)  $ and orders $q$. The forbidden
pairs $\left(  a,b\right)  $ for $\mathbb{F}_{\left(  m,n\right)  }^{\left[
a,b\right]  }\left(  q\right)  $ coincide with ones for polyadic rings
$\mathbb{Z}_{\left(  m,n\right)  }^{\left[  a,b\right]  }$ listed in (\ref{b}).

\begin{table}[pb]
\caption[Content and arities of finite polyadic rings and fields]{{Content and
arities of the secondary class finite polyadic rings }$\mathbb{Z}_{\left(
m,n\right)  }^{\left[  a,b\right]  }\left(  q\right)  ${ and the corresponding
simple finite polyadic $\left(  m,n\right)  $-fields $\mathbb{F}%
_{m,n}^{\left[  a,b\right]  }\left(  q\right)  $ of order $2\leq q\leq4$
(framed) for $2\leq b\leq6$. The subscripts $e$ and $z$ mark those secondary
classes (two primes are omitted) which play the roles of polyadic unit and
polyadic zero respectively. The double frames denote the finite polyadic
fields containing both a unit and zero. The last line in a cell (corresponding
to a fixed congruence class $\left[  \left[  a\right]  \right]  _{b}$) gives
the allowed orders of finite polyadic fields for $5\leq q\leq10$, and bold
numbers mark the orders of such fields which contain both unit(s) and zero.}}%
\label{T1}
\begin{center}%
\begin{tabular}
[c]{||c||c|c|c|c|c|}\hline\hline
$a\setminus b$ & 2 & 3 & 4 & 5 & 6\\\hline\hline
1 & $%
\begin{array}
[c]{c}%
m=\mathbf{3}\\
n=\mathbf{2}\\
\fbox{\emph{1}$_{e}$\emph{,3}}\\
\fbox{\fbox{\emph{1}$_{e}$\emph{,3}$_{z}$\emph{,5}}}\\
\fbox{\emph{1}$_{e}$\emph{,3,5,7}}\\
\text{\textsf{q=\textbf{5},\textbf{7},8}}%
\end{array}
$ & $%
\begin{array}
[c]{c}%
m=\mathbf{4}\\
n=\mathbf{2}\\
\fbox{\fbox{\emph{1}$_{e}$\emph{,4}$_{z}$}}\\
\fbox{\emph{1}$_{e}$\emph{,4,7}}\\
\text{\emph{1}}_{e}\text{\emph{,4}}_{z}\text{\emph{,7,10}}\\
\text{\textsf{q=\textbf{5},\textbf{7},9}}%
\end{array}
$ & $%
\begin{array}
[c]{c}%
m=\mathbf{5}\\
n=\mathbf{2}\\
\fbox{\emph{1}$_{e}$\emph{,5}}\\
\fbox{\fbox{\emph{1}$_{e}$\emph{,5,9}$_{z}$}}\\
\fbox{\emph{1}$_{e}$\emph{,5,9,13}}\\
\text{\textsf{q=\textbf{5},\textbf{7},8}}%
\end{array}
$ & $%
\begin{array}
[c]{c}%
m=\mathbf{6}\\
n=\mathbf{2}\\
\fbox{\fbox{\emph{1}$_{e}$\emph{,6}$_{z}$}}\\
\fbox{\fbox{\emph{1}$_{e}$\emph{,6}$_{z}$\emph{,11}}}\\
\emph{1}_{e}\emph{,6,11,16}_{z}\\
\text{\textsf{q=5,\textbf{7}}}%
\end{array}
$ & $%
\begin{array}
[c]{c}%
m=\mathbf{7}\\
n=\mathbf{2}\\
\fbox{\emph{1}$_{e}$\emph{,7}}\\
\fbox{\emph{1}$_{e}$\emph{,7,13}}\\
\fbox{\emph{1}$_{e}$\emph{,7,13,19}}\\
\text{\textsf{q=\textbf{5},6,\textbf{7},8,9}}%
\end{array}
$\\\hline
2 &  & $%
\begin{array}
[c]{c}%
m=\mathbf{4}\\
n=\mathbf{3}\\
\fbox{\fbox{\emph{2}$_{z}$\emph{,5}$_{e}$}}\\
\fbox{\emph{2,5,8}$_{e}$}\\
\text{\emph{2,5}}_{e}\text{\emph{,8}}_{z}\text{\emph{,11}}_{e}\\
\text{\textsf{q=\textbf{5},\textbf{7},9}}%
\end{array}
$ &  & $%
\begin{array}
[c]{c}%
m=\mathbf{6}\\
n=\mathbf{5}\\
\fbox{\fbox{\emph{2}$_{z}$\emph{,7}$_{e}$}}\\
\fbox{\fbox{\emph{2}$_{e}$\emph{,7,12}$_{z}$}}\\
\text{\emph{2,7}}_{e}\text{\emph{,12}}_{z}\text{\emph{,17}}_{e}\\
\text{\textsf{q=5,\textbf{7}}}%
\end{array}
$ & $%
\begin{array}
[c]{c}%
m=\mathbf{4}\\
n=\mathbf{3}\\
\text{\emph{2,8}}_{z}\\
\fbox{\emph{2,8}$_{e}$\emph{,14}}\\
\text{\emph{2,8}}_{z}\text{\emph{,14,20}}\\
\text{\textsf{q=\textbf{5},\textbf{7},9}}%
\end{array}
$\\\hline
3 &  &  & $%
\begin{array}
[c]{c}%
m=\mathbf{5}\\
n=\mathbf{3}\\
\fbox{\emph{3}$_{e}$\emph{,7}$_{e}$}\\
\fbox{\fbox{\emph{3}$_{z}$\emph{,7}$_{e}$,\emph{11}$_{e}$}}\\
\fbox{\emph{3,7}$_{e}$\emph{,11,15}$_{e}$}\\
\text{\textsf{q=\textbf{5},6,\textbf{7},8}}%
\end{array}
$ & $%
\begin{array}
[c]{c}%
m=\mathbf{6}\\
n=\mathbf{5}\\
\fbox{\fbox{\emph{3}$_{e}$\emph{,8}$_{z}$}}\\
\fbox{\fbox{\emph{3}$_{z}$\emph{,8}$_{e}$\emph{,13}$_{e}$}}\\
\text{\emph{3}}_{e}\text{\emph{,8}}_{z}\text{\emph{,13}}_{e}\text{\emph{,18}%
}\\
\text{\textsf{q=5,\textbf{7}}}%
\end{array}
$ & $%
\begin{array}
[c]{c}%
m=\mathbf{3}\\
n=\mathbf{2}\\
\fbox{\emph{3,9}$_{e}$}\\
\text{\emph{3,9}}_{z}\text{\emph{,15}}\\
\fbox{\emph{3,9}$_{e}$\emph{,15,21}}\\
\text{\textsf{q=\textbf{5},\textbf{7},8}}%
\end{array}
$\\\hline
4 &  &  &  & $%
\begin{array}
[c]{c}%
m=\mathbf{6}\\
n=\mathbf{3}\\
\fbox{\fbox{\emph{4}$_{z}$\emph{,9}$_{e}$}}\\
\fbox{\fbox{\emph{4}$_{e}$\emph{,9}$_{z}$\emph{,14}$_{e}$}}\\
\text{\emph{4}}_{z}\text{\emph{,9}}_{e}\text{\emph{,14,19}}_{e}\\
\text{\textsf{q=5,\textbf{7}}}%
\end{array}
$ & $%
\begin{array}
[c]{c}%
m=\mathbf{4}\\
n=\mathbf{2}\\
\text{\emph{4}}_{z}\text{\emph{,10}}\\
\fbox{\emph{4,10}$_{e}$\emph{,16}}\\
\text{\emph{4,10,16}}_{z}\text{\emph{,22}}\\
\text{\textsf{q=\textbf{5,7},9}}%
\end{array}
$\\\hline
5 &  &  &  &  & $%
\begin{array}
[c]{c}%
m=\mathbf{7}\\
n=\mathbf{3}\\
\fbox{\emph{5}$_{e}$\emph{,11}$_{e}$}\\
\fbox{\emph{5,11,17}$_{e}$}\\
\fbox{\emph{5}$_{e}$\emph{,11}$_{e}$\emph{,17}$_{e}$\emph{,23}$_{e}$}\\
\text{\textsf{q=\textbf{5},6,\textbf{7},8,9}}%
\end{array}
$\\\hline\hline
\end{tabular}
\end{center}
\end{table}

\subsection{Multiplicative structure}

In the multiplicative structure the following crucial differences between the
binary finite fields and $\mathbb{F}_{\left(  m,n\right)  }\left(  q\right)  $
can be outlined.

\begin{remark}
The order of a \textsl{non-extended} finite polyadic field may not be
\textsl{prime} (e.g., $\mathbb{F}_{\left(  3,2\right)  }^{\left[  1,2\right]
}\left(  4\right)  $, $\mathbb{F}_{\left(  5,3\right)  }^{\left[  3,4\right]
}\left(  8\right)  $, $\mathbb{F}_{\left(  4,3\right)  }^{\left[  2,6\right]
}\left(  9\right)  $), and may not even be a power of a prime binary number
(e.g. $\mathbb{F}_{\left(  7,3\right)  }^{\left[  5,6\right]  }\left(
6\right)  $, $\mathbb{F}_{\left(  11,5\right)  }^{\left[  3,10\right]
}\left(  10\right)  $), see \textsc{Table \ref{T2}}.
\end{remark}

\begin{remark}
The polyadic characteristic $\chi_{p}$ of a non-extended finite polyadic field
can have values such that $\chi_{p}+1$ (corresponding in the binary case to
the ordinary characteristic $\chi$) can be \textsl{nonprime} (\textsc{Table
\ref{T0}}).
\end{remark}

\begin{assertion}
If a secondary class $\left[  \left[  x_{k}^{\left[  a,b\right]  }\right]
\right]  _{bq}$ contains \textsl{no zero}, it can be isomorphic to the
\textquotedblleft abstract\textquotedblright\ \textit{zeroless} finite
polyadic field $\mathbb{F}_{\left(  m,n\right)  }^{\left[  a,b\right]
}\left(  q\right)  $.
\end{assertion}

Zeroless fields are marked by one frame in \textsc{Table \ref{T1}}. There
exist finite polyadic fields with \textsl{more than one} unit, and also
\textsl{all} elements can be units. Such cases are marked in \textsc{Table
\ref{T2}} by subscripts which indicate the number of units.

Denote the Abelian finite multiplicative $n$-ary group of $\mathbb{F}_{\left(
m,n\right)  }^{\left[  a,b\right]  }\left(  q\right)  $ by $\mathcal{G}%
_{n}^{\left[  a,b\right]  }\left(  q\right)  $.

\begin{example}
\textbf{1)} The finite $\left(  6,3\right)  $-field $\mathbb{F}_{\left(
6,3\right)  }^{\left[  4,5\right]  }\left(  3\right)  $ of order $3$ has
\textsl{two units} $\left\{  4,14\right\}  \equiv\left\{  4_{e},14_{e}%
\right\}  $ and zero $9\equiv9_{z}$, and its multiplicative $3$-ary group
$\mathcal{G}_{3}^{\left[  4,5\right]  }\left(  2\right)  $ is%
\[
\mu_{3}\left[  4,4,4\right]  =4,\ \mu_{3}\left[  4,4,14\right]  =14,\ \mu
_{3}\left[  4,14,14\right]  =4,\ \mu_{3}\left[  14,14,14\right]  =14.
\]

\textbf{2)} In $\mathbb{F}_{\left(  7,3\right)  }^{\left[  5,6\right]
}\left(  4\right)  $ of order $4$ \textsl{all} the elements $\left\{
5,11,17,23\right\}  \equiv\left\{  5_{e},11_{e},17_{e},23_{e}\right\}  $ are
units (see (\ref{e})), because for its $3$-ary group $\mathcal{G}_{3}^{\left[
5,6\right]  }\left(  4\right)  $ we have%
\begin{align*}
\mu_{3}\left[  5,5,5\right]   &  =5,\mu_{3}\left[  11,11,11\right]
=11,\mu_{3}\left[  17,17,17\right]  =17,\mu_{3}\left[  23,23,23\right]  =23,\\
\mu_{3}\left[  5,5,11\right]   &  =11,\ \mu_{3}\left[  5,5,17\right]
=17,\ \mu_{3}\left[  5,5,23\right]  =23,\\
\mu_{3}\left[  11,11,5\right]   &  =5,\ \mu_{3}\left[  11,11,17\right]
=17,\ \mu_{3}\left[  11,11,23\right]  =23,\\
\mu_{3}\left[  17,17,5\right]   &  =5,\ \mu_{3}\left[  17,17,11\right]
=11,\ \mu_{3}\left[  17,17,23\right]  =23,\\
\mu_{3}\left[  23,23,5\right]   &  =5,\ \mu_{3}\left[  23,23,11\right]
=11,\ \mu_{3}\left[  23,23,17\right]  =17.
\end{align*}

\end{example}

In general, $n$-ary groups may contain no units (and multiplicative
idempotents) at all, and invertibility is controlled in another way, by
querelements: each element of any $n$-ary group should be (uniquely)
\textquotedblleft quereable\textquotedblright\ (\ref{nm}). In case of $\left(
m,n\right)  $-fields both $m$-ary additive group and $n$-ary multiplicative
group $\mathcal{G}_{n}^{\left[  a,b\right]  }\left(  q\right)  $ can be of
this kind. By analogy with zeroless-nonunital rings we have

\begin{definition}
A polyadic field $\mathbb{F}_{\left(  m,n\right)  }$ is called
\textit{zeroless-nonunital}, if it contains no zero and no unit.
\end{definition}

\begin{assertion}
The zeroless-nonunital polyadic fields are totally (additively and
multiplicatively) nonderived.
\end{assertion}

\begin{proposition}
\textbf{1)} If $\gcd\left(  bq\right)  \neq1$, then $\mathbb{F}_{\left(
m,n\right)  }^{\left[  a,b\right]  }\left(  q\right)  $ is zeroless.
\textbf{2)} The zero can exist only, if $\gcd\left(  bq\right)  =1$ and the
field order $q=p$ is prime.
\end{proposition}

\begin{proof}
It follows directly from the definition of the polyadic zero (\ref{z}) and
(\ref{kadd}).
\end{proof}

Let us consider examples of zeroless-nonunital finite fields of polyadic
integer numbers $\mathbb{F}_{\left(  m,n\right)  }^{\left[  a,b\right]
}\left(  q\right)  $.

\begin{example}
\label{exam-fld}\textbf{1)} The zeroless-nonunital polyadic finite fields
having lowest $\left\vert a+b\right\vert $ are, e.g., $\mathbb{F}_{\left(
9,3\right)  }^{\left[  3,8\right]  }\left(  2\right)  $, $\mathbb{F}_{\left(
9,3\right)  }^{\left[  3,8\right]  }\left(  4\right)  $, $\mathbb{F}_{\left(
9,3\right)  }^{\left[  5,8\right]  }\left(  4\right)  $, $\mathbb{F}_{\left(
9,3\right)  }^{\left[  5,8\right]  }\left(  8\right)  $, also $\mathbb{F}%
_{\left(  10,4\right)  }^{\left[  4,9\right]  }\left(  3\right)  $,
$\mathbb{F}_{\left(  10,4\right)  }^{\left[  4,9\right]  }\left(  9\right)  $,
and $\mathbb{F}_{\left(  10,4\right)  }^{\left[  7,9\right]  }\left(
3\right)  $, $\mathbb{F}_{\left(  10,4\right)  }^{\left[  7,9\right]  }\left(
9\right)  $.

\textbf{2)} The multiplication of the zeroless-nonunital $\left(  9,3\right)
$-field $\mathbb{F}_{\left(  9,3\right)  }^{\left[  5,8\right]  }\left(
2\right)  $ is%
\[
\mu_{3}\left[  5,5,5\right]  =13,\ \mu_{3}\left[  5,5,13\right]  =5,\ \mu
_{3}\left[  5,13,13\right]  =13,\ \mu_{3}\left[  13,13,13\right]  =5.
\]

Using (\ref{xx}) we find the (unique) multiplicative querelements $\bar{5}%
=13$, $\overline{13}=5$. The addition of $\mathbb{F}_{\left(  9,3\right)
}^{\left[  5,8\right]  }\left(  2\right)  $ is%
\begin{align*}
\nu_{9}\left[  5^{9}\right]   &  =13,\ \nu_{9}\left[  5^{8},13\right]
=5,\ \nu_{9}\left[  5^{7},13^{2}\right]  =13,\ \nu_{9}\left[  5^{6}%
,13^{3}\right]  =5,\ \nu_{9}\left[  5^{5},13^{4}\right]  =13,\\
\ \nu_{9}\left[  5^{4},13^{5}\right]   &  =5,\ \nu_{9}\left[  5^{3}%
,13^{6}\right]  =13,\ \nu_{9}\left[  5^{2},13^{7}\right]  =5,\ \nu_{9}\left[
5,13^{8}\right]  =13,\nu_{9}\left[  13^{9}\right]  =5.
\end{align*}

The additive (unique) querelements are $\tilde{5}=13$, $\widetilde{13}=5$. So
all elements are additively and multiplicatively querable (polyadically
invertible), and therefore $\nu_{9}$ is $9$-ary additive group operation and
$\mu_{3}$ is $3$-ary multiplicative group operation, as it should be for a
field. Because it contains no unit and no zero, $\mathbb{F}_{\left(
9,3\right)  }^{\left[  5,8\right]  }\left(  2\right)  $ is actually a
zeroless-nonunital finite $\left(  9,3\right)  $-field of order $2$.
\end{example}

Other zeroless-nonunital finite polyadic fields are marked by frames in
\textsc{Table \ref{T2}}.

\begin{remark}
The absence of zero \textsl{does not} guarantee that a $\left(  m,n\right)
$-ring $\mathcal{R}_{\left(  m,n\right)  }^{\left[  a,b\right]  }\left(
q\right)  $ is a field. For that, both $\left\langle \left[  \left[  a\right]
\right]  _{b}\mid\nu_{m}\right\rangle $ and $\left\langle \left[  \left[
a\right]  \right]  _{b}\mid\mu_{n}\right\rangle $ have to be polyadic groups.
\end{remark}

\begin{example}
The $\left(  4,3\right)  $-ring $\mathcal{R}_{\left(  4,3\right)  }^{\left[
2,3\right]  }\left(  6\right)  $ is zeroless, and $\left\langle \left[
\left[  3\right]  \right]  _{4}\mid\nu_{4}\right\rangle $ is its $4$-ary
additive group (each element has a unique additive querelement). Despite each
element of $\left\langle \left[  \left[  2\right]  \right]  _{3}\mid\mu
_{3}\right\rangle $ having a querelement, it is not a multiplicative $3$-ary
group, because for the two elements $2$ and $14$ we have \textsl{nonunique}
querelements
\begin{equation}
\mu_{3}\left[  2,2,5\right]  =2,\ \mu_{3}\left[  2,2,14\right]  =2,\ \mu
_{3}\left[  14,14,2\right]  =14,\ \mu_{3}\left[  14,14,11\right]  =14.
\end{equation}

\end{example}

The conditions on the congruence classes $\left[  \left[  a\right]  \right]
_{b}$ and the invariants $I,J$ (\ref{ij}), which give the same arity structure
are given in \cite{dup2017}. Note, that there exist polyadic fields of the
same arities $\left(  m,n\right)  $ and the same order $q$ which are not
isomorphic (in contrast with what is possible in the binary case).

\begin{example}
The polyadic $\left(  9,3\right)  $-fields corresponding to the congruence
classes $\left[  \left[  5\right]  \right]  _{8}$ and $\left[  \left[
7\right]  \right]  _{8}$ are not isomorphic for orders $q=2,4,8$ (see
\textsc{Table \ref{T2}}). Despite both being zeroless, the first
$\mathbb{F}_{\left(  9,3\right)  }^{\left[  5,8\right]  }\left(  q\right)  $
are nonunital, while the second $\mathbb{F}_{\left(  9,3\right)  }^{\left[
7,8\right]  }\left(  q\right)  $ has two units, which makes an isomorphism impossible.
\end{example}

Recall \cite{lid/nie}, that in a (non-extended, prime) finite binary field
$\mathbb{F}\left(  p\right)  $, the \textit{order of an element}
$x\in\mathbb{F}\left(  p\right)  $ is defined as a smallest integer $\lambda$
such that $x^{\lambda}=1$. Obviously, the set of fixed order elements forms a
cyclic subgroup $\mathcal{G}_{\lambda}$ of the multiplicative binary group of
$\mathbb{F}\left(  p\right)  $, and $\lambda\mid\left(  p-1\right)  $. If
$\lambda=p-1$, such an element is called a \textit{primitive (root)}, it
generates all elements, and these exist in any finite binary field. Moreover,
any element of $\mathbb{F}\left(  p\right)  $ is idempotent $x^{p}=x$, while
all its nonzero elements satisfy $x^{p-1}=1$ (Fermat's little theorem). A
non-extended (prime) finite field is fully determined by its order $p$ (up to
isomorphism), and, moreover, any $\mathbb{F}\left(  p\right)  $ is isomorphic
to $\mathbb{Z}\diagup p\mathbb{Z}$.

In the polyadic case, the situation is more complicated. Because the related
secondary class structure (\ref{fmn}) contains parameters in addition to the
number of elements $q$, the order of (non-extended) polyadic fields may not be
prime, or nor even a power of a prime integer (e.g. $\mathbb{F}_{\left(
7,3\right)  }^{\left[  5,6\right]  }\left(  6\right)  $ or $\mathbb{F}%
_{\left(  11,5\right)  }^{\left[  3,10\right]  }\left(  10\right)  $). Also,
as was shown above, finite polyadic fields can be zeroless, nonunital and have
many (or even all) units (see \textsc{Table \ref{T2}}). Therefore, we cannot
use units in the definition of the element order. Instead, we propose an alternative:

\begin{definition}
If an element of the finite polyadic field $x\in\mathbb{F}_{\left(
m,n\right)  }\left(  q\right)  $ satisfies%
\begin{equation}
x^{\left\langle \lambda_{p}\right\rangle _{\times n}}=x,
\end{equation}
then the smallest such $\lambda_{p}$ is called the \textit{idempotence
polyadic order} and denoted $\operatorname*{ord}x=\lambda_{p}$.
\end{definition}

Obviously, $\lambda_{p}=\lambda$ (see (\ref{xp})).

\begin{definition}
The \textit{idempotence polyadic order} $\lambda_{p}^{\left[  a,b\right]  }$
of a finite polyadic field $\mathbb{F}_{\left(  m,n\right)  }^{\left[
a,b\right]  }\left(  q\right)  $ is the maximum $\lambda_{p}$ of all its
elements, and we call such field $\lambda_{p}^{\left[  a,b\right]  }%
$-\textit{idempotent }and denote $\operatorname*{ord}\mathbb{F}_{\left(
m,n\right)  }^{\left[  a,b\right]  }\left(  q\right)  =\lambda_{p}^{\left[
a,b\right]  }$.
\end{definition}

In \textsc{Table \ref{T2}} we present the idempotence polyadic order
$\lambda_{p}^{\left[  a,b\right]  }$ for the (allowed) finite polyadic fields
$\mathbb{F}_{\left(  m,n\right)  }^{\left[  a,b\right]  }\left(  q\right)  $
(\ref{fmn}) with $2\leq b\leq10$ and order $q\leq10$.

\begin{definition}
Denote by $q_{\ast}$ the number of nonzero \textsl{distinct} elements in
$\mathbb{F}_{\left(  m,n\right)  }\left(  q\right)  $%
\begin{equation}
q_{\ast}=\left\{
\begin{array}
[c]{c}%
q-1,\ \ \ \text{if }\exists z\in\mathbb{F}_{\left(  m,n\right)  }\left(
q\right) \\
q,\ \ \ \text{if }\nexists z\in\mathbb{F}_{\left(  m,n\right)  }\left(
q\right)  ,
\end{array}
\right.  \label{qs}%
\end{equation}
which is called a \textit{reduced (field) order}.
\end{definition}

The second choice of (\ref{qs}) in the binary case is absent, because any
commutative binary group (as the additive group of a field) contains a zero
(the identity of this group), and therefore any binary field has a zero, which
does not always hold for the $m$-ary additive group of $\mathbb{F}_{\left(
m,n\right)  }$ (see \emph{Example \ref{exam-fld}}).

\begin{table}[pbh]
\caption[Idempotence polyadic orders for the (allowed) finite polyadic
fields]{Idempotence polyadic orders $\lambda_{p}^{\left[  a,b\right]  }$ for
(allowed) finite polyadic fields $\mathbb{F}_{\left(  m,n\right)  }^{\left[
a,b\right]  }\left(  q\right)  $, with $2\leq b\leq10$ and field orders $2\leq
q\leq10$. If $q_{\ast}$ is $n$-admissible, they are underlined. The
zeroless-nonunital cases are framed. Subscripts correspond to the number of
units $\kappa_{e}\geq2$, while nonframed entries have exactly one unit.}%
\label{T2}
\begin{center}%
\begin{tabular}
[c]{||c|c||c||c|c|c|c|c|c|c|c|c||}\hline\hline
\multicolumn{2}{||c||}{$\left[  \left[  a\right]  \right]  _{b}$} &
\emph{Arities} & \multicolumn{9}{|c||}{\emph{Finite polyadic field order }$q$%
}\\\hline
$b$ & $a$ & $(m,n)$ & \textsl{2} & \textsl{3} & \textsl{4} & \textsl{5} &
\textsl{6} & \textsl{7} & \textsl{8} & \textsl{9} & \textsl{10}\\\hline\hline
\textsf{2} & \textsf{1} & $\left(  3,2\right)  $ & $2$ & $\mathbf{2}$ & $2$ &
$\mathbf{4}$ & $\varnothing$ & $\mathbf{6}$ & $4$ & $\varnothing$ &
$\varnothing$\\\hline
\textsf{3} & \textsf{1} & $\left(  4,2\right)  $ & $\mathbf{1}$ & $3$ &
$\varnothing$ & $\mathbf{4}$ & $\varnothing$ & $\mathbf{6}$ & $\varnothing$ &
$9$ & $\varnothing$\\\hline
& \textsf{2} & $\left(  4,3\right)  $ & $\mathbf{1}$ & \underline{$3$} &
$\varnothing$ & $\mathbf{2}_{2e}$ & $\varnothing$ & $\mathbf{3}_{2e}$ &
$\varnothing$ & \underline{$9$} & $\varnothing$\\\hline
\textsf{4} & \textsf{1} & $\left(  5,2\right)  $ & $2$ & $\mathbf{2}$ & $4$ &
$\mathbf{4}$ & $\varnothing$ & $\mathbf{6}$ & $8$ & $\varnothing$ &
$\varnothing$\\\hline
& \textsf{3} & $\left(  5,3\right)  $ & $1_{2e}$ & $\mathbf{1}_{2e}$ &
$2_{2e}$ & $\mathbf{2}_{2e}$ & $\varnothing$ & $\mathbf{3}_{2e}$ & $4_{2e}$ &
$\varnothing$ & $\varnothing$\\\hline
\textsf{5} & \textsf{1} & $\left(  6,2\right)  $ & $\mathbf{1}$ & $\mathbf{2}$
& $\varnothing$ & $5$ & $\varnothing$ & $\mathbf{6}$ & $\varnothing$ &
$\varnothing$ & $\varnothing$\\\hline
& \textsf{2} & $\left(  6,5\right)  $ & $\mathbf{1}$ & $\mathbf{1}_{2e}$ &
$\varnothing$ & \underline{$5$} & $\varnothing$ & $\mathbf{3}_{2e}$ &
$\varnothing$ & $\varnothing$ & $\varnothing$\\\hline
& \textsf{3} & $\left(  6,5\right)  $ & $\mathbf{1}$ & $\mathbf{1}_{2e}$ &
$\varnothing$ & \underline{$5$} & $\varnothing$ & $\mathbf{3}_{2e}$ &
$\varnothing$ & $\varnothing$ & $\varnothing$\\\hline
& \textsf{4} & $\left(  6,3\right)  $ & $\mathbf{1}$ & $\mathbf{1}_{2e}$ &
$\varnothing$ & \underline{$5$} & $\varnothing$ & $\mathbf{3}_{2e}$ &
$\varnothing$ & $\varnothing$ & $\varnothing$\\\hline
\textsf{6} & \textsf{1} & $\left(  7,2\right)  $ & $2$ & $3$ & $2$ &
$\mathbf{4}$ & $6$ & $\mathbf{6}$ & $4$ & $9$ & $\varnothing$\\\hline
& \textsf{2} & $\left(  4,3\right)  $ & $\varnothing$ & \underline{$3$} &
$\varnothing$ & $\mathbf{2}_{2e}$ & $\varnothing$ & $\mathbf{3}_{2e}$ &
$\varnothing$ & \underline{$9$} & $\varnothing$\\\hline
& \textsf{3} & $\left(  3,2\right)  $ & $2$ & $\varnothing$ & $2$ &
$\mathbf{4}$ & $\varnothing$ & $\mathbf{6}$ & $4$ & $\varnothing$ &
$\varnothing$\\\hline
& \textsf{4} & $\left(  4,2\right)  $ & $\varnothing$ & $3$ & $\varnothing$ &
$\mathbf{4}$ & $\varnothing$ & $\mathbf{6}$ & $\varnothing$ & $9$ &
$\varnothing$\\\hline
& \textsf{5} & $\left(  7,3\right)  $ & $1_{2e}$ & \underline{$3$} & $1_{4e}$
& $\mathbf{2}_{2e}$ & $3_{2e}$ & $\mathbf{3}_{2e}$ & $2$ & \underline{$9$} &
$\varnothing$\\\hline
\textsf{7} & \textsf{1} & $\left(  8,2\right)  $ & $\mathbf{1}$ & $\mathbf{2}$
& $\varnothing$ & $\mathbf{4}$ & $\varnothing$ & $7$ & $\varnothing$ &
$\varnothing$ & $\varnothing$\\\hline
& \textsf{2} & $\left(  8,4\right)  $ & $\mathbf{1}$ & $\mathbf{2}$ &
$\varnothing$ & \underline{$\mathbf{4}$} & $\varnothing$ & \underline{$7$} &
$\varnothing$ & $\varnothing$ & $\varnothing$\\\hline
& \textsf{3} & $\left(  8,7\right)  $ & $\mathbf{1}$ & $\mathbf{1}_{2e}$ &
$\varnothing$ & $\mathbf{2}_{2e}$ & $\varnothing$ & \underline{$7$} &
$\varnothing$ & $\varnothing$ & $\varnothing$\\\hline
& \textsf{4} & $\left(  8,4\right)  $ & $\mathbf{1}$ & $\mathbf{2}$ &
$\varnothing$ & \underline{$\mathbf{4}$} & $\varnothing$ & \underline{$7$} &
$\varnothing$ & $\varnothing$ & $\varnothing$\\\hline
& \textsf{5} & $\left(  8,7\right)  $ & $\mathbf{1}$ & $\mathbf{1}_{2e}$ &
$\varnothing$ & $\mathbf{2}_{2e}$ & $\varnothing$ & \underline{$7$} &
$\varnothing$ & $\varnothing$ & $\varnothing$\\\hline
& \textsf{6} & $\left(  8,3\right)  $ & $\mathbf{1}$ & $\mathbf{1}_{2e}$ &
$\varnothing$ & $\mathbf{2}_{2e}$ & $\varnothing$ & \underline{$7$} &
$\varnothing$ & $\varnothing$ & $\varnothing$\\\hline
\textsf{8} & \textsf{1} & $\left(  9,2\right)  $ & $2$ & $\mathbf{2}$ & $4$ &
$\mathbf{4}$ & $\varnothing$ & $\mathbf{6}$ & $8$ & $\varnothing$ &
$\varnothing$\\\hline
& \textsf{3} & $\left(  9,3\right)  $ & \fbox{$2$} & $\mathbf{1}_{2e}$ &
\fbox{$4$} & $\mathbf{2}_{2e}$ & $\varnothing$ & $\mathbf{3}_{2e}$ &
\fbox{$8$} & $\varnothing$ & $\varnothing$\\\hline
& \textsf{5} & $\left(  9,3\right)  $ & \fbox{$2$} & $\mathbf{1}_{2e}$ &
\fbox{$4$} & $\mathbf{2}_{2e}$ & $\varnothing$ & $\mathbf{3}_{2e}$ &
\fbox{$8$} & $\varnothing$ & $\varnothing$\\\hline
& \textsf{7} & $\left(  9,3\right)  $ & $1_{2e}$ & $\mathbf{1}_{2e}$ &
$2_{2e}$ & $\mathbf{2}_{2e}$ & $\varnothing$ & $\mathbf{3}_{2e}$ & $4_{2e}$ &
$\varnothing$ & $\varnothing$\\\hline
\textsf{9} & \textsf{1} & $\left(  10,2\right)  $ & $\mathbf{1}$ & $3$ &
$\varnothing$ & $\mathbf{4}$ & $\varnothing$ & $\mathbf{6}$ & $\varnothing$ &
$9$ & $\varnothing$\\\hline
& \textsf{2} & $\left(  10,7\right)  $ & $\mathbf{1}$ & \fbox{$3$} &
$\varnothing$ & $\mathbf{2}_{2e}$ & $\varnothing$ & $\mathbf{1}_{6e}$ &
$\varnothing$ & \fbox{$9$} & $\varnothing$\\\hline
& \textsf{4} & $\left(  10,4\right)  $ & $\mathbf{1}$ & \fbox{$3$} &
$\varnothing$ & \underline{$\mathbf{4}$} & $\varnothing$ & $\mathbf{2}_{3e}$ &
$\varnothing$ & \fbox{$9$} & $\varnothing$\\\hline
& \textsf{5} & $\left(  10,7\right)  $ & $\mathbf{1}$ & \fbox{$3$} &
$\varnothing$ & $\mathbf{2}_{2e}$ & $\varnothing$ & $\mathbf{1}_{6e}$ &
$\varnothing$ & \fbox{$9$} & $\varnothing$\\\hline
& \textsf{7} & $\left(  10,4\right)  $ & $\mathbf{1}$ & \fbox{$3$} &
$\varnothing$ & \underline{$\mathbf{4}$} & $\varnothing$ & $\mathbf{2}_{3e}$ &
$\varnothing$ & \fbox{$9$} & $\varnothing$\\\hline
& \textsf{8} & $\left(  10,3\right)  $ & $\mathbf{1}$ & $3$ & $\varnothing$ &
$\mathbf{2}_{2e}$ & $\varnothing$ & $\mathbf{3}_{2e}$ & $\varnothing$ & $9$ &
$\varnothing$\\\hline
\textsf{10} & \textsf{1} & $\left(  11,2\right)  $ & $2$ & $\mathbf{2}$ & $2$
& $5$ & $\varnothing$ & $\mathbf{6}$ & $8$ & $\varnothing$ & $10$\\\hline
& \textsf{2} & $\left(  6,5\right)  $ & $\varnothing$ & $\mathbf{1}_{2e}$ &
$\varnothing$ & \underline{$5$} & $\varnothing$ & $\mathbf{3}_{2e}$ &
$\varnothing$ & $\varnothing$ & $\varnothing$\\\hline
& \textsf{3} & $\left(  11,5\right)  $ & $1_{2e}$ & $\mathbf{1}_{2e}$ &
$1_{4e}$ & \underline{$5$} & $\varnothing$ & $\mathbf{3}_{2e}$ & $1_{8e}$ &
$\varnothing$ & $5_{2e}$\\\hline
& \textsf{4} & $\left(  6,3\right)  $ & $\varnothing$ & $\mathbf{1}_{2e}$ &
$\varnothing$ & \underline{$5$} & $\varnothing$ & $\mathbf{3}_{2e}$ &
$\varnothing$ & $\varnothing$ & $\varnothing$\\\hline
& \textsf{5} & $\left(  3,2\right)  $ & $2$ & $\mathbf{2}$ & $2$ &
$\varnothing$ & $\varnothing$ & $\mathbf{6}$ & $\varnothing$ & $\varnothing$ &
$\varnothing$\\\hline
& \textsf{6} & $\left(  6,2\right)  $ & $\varnothing$ & $\mathbf{2}$ &
$\varnothing$ & $5$ & $\varnothing$ & $\mathbf{6}$ & $\varnothing$ &
$\varnothing$ & $\varnothing$\\\hline
& \textsf{7} & $\left(  11,5\right)  $ & $1_{2e}$ & $\mathbf{1}_{2e}$ &
$1_{4e}$ & \underline{$5$} & $\varnothing$ & $\mathbf{3}_{2e}$ & $1_{8e}$ &
$\varnothing$ & $5_{2e}$\\\hline
& \textsf{8} & $\left(  6,5\right)  $ & $\varnothing$ & $\mathbf{1}_{2e}$ &
$\varnothing$ & \underline{$5$} & $\varnothing$ & $\mathbf{3}_{2e}$ &
$\varnothing$ & $\varnothing$ & $\varnothing$\\\hline
& \textsf{9} & $\left(  11,3\right)  $ & $1_{2e}$ & $\mathbf{1}_{2e}$ &
$1_{4e}$ & \underline{$5$} & $\varnothing$ & $\mathbf{3}_{2e}$ & $2_{4e}$ &
$\varnothing$ & $5_{2e}$\\\hline\hline
\end{tabular}
\end{center}
\end{table}

\begin{theorem}
If a finite polyadic field $\mathbb{F}_{\left(  m,n\right)  }\left(  q\right)
$ has an order $q$, such that $q_{\ast}=q_{\ast}^{adm}=\ell\left(  n-1\right)
+1$ is $n$-admissible, then (for $n\geq3$ and one unit):

\begin{enumerate}
\item A sequence of the length $q_{\ast}\left(  n-1\right)  $ built from any
fixed element $y\in\mathbb{F}_{\left(  m,n\right)  }\left(  q\right)  $ is
neutral%
\begin{equation}
\mu_{n}^{\left(  q_{\ast}\right)  }\left[  x,y^{q_{\ast}\left(  n-1\right)
}\right]  =x,\ \ \ \forall x\in\mathbb{F}_{\left(  m,n\right)  }\left(
q\right)  . \label{mx}%
\end{equation}

\item Any element $y$ satisfies the polyadic idempotency condition%
\begin{equation}
y^{\left\langle q_{\ast}\right\rangle _{\times n}}=y,\ \ \ \forall
y\in\mathbb{F}_{\left(  m,n\right)  }\left(  q\right)  . \label{yy}%
\end{equation}

\end{enumerate}
\end{theorem}

\begin{proof}
\textbf{1)} Take a long $n$-ary product of the $q_{\ast}$ distinct nonzero
elements $x_{0}=\mu_{n}^{\left(  \ell\right)  }\left[  x_{1},x_{2}%
,\ldots,x_{q_{_{\ast}}}\right]  $, such that $q_{_{\ast}}$ can take only
multiplicatively $n$-admissible\ values $q_{\ast}^{adm}$, where $\ell
\in\mathbb{N}$ is a \textquotedblleft number\textquotedblright\ of $n$-ary
multiplications. Then polyadically multiply each $x_{i}$ by a fixed element
$y\in\mathbb{F}_{\left(  m,n\right)  }\left(  q\right)  $ such that all
$q_{_{\ast}}$ elements $\mu_{n}\left[  x_{i},y^{n-1}\right]  $ will be
distinct as well. Therefore, their product should be the same $x_{0}$. Using
commutativity and associativity, we obtain%
\begin{align}
x_{0}  &  =\mu_{n}^{\left(  \ell\right)  }\left[  x_{1},x_{2},\ldots
,x_{q_{_{\ast}}}\right]  =\mu_{n}^{\left(  \ell\right)  }\left[  \mu
_{n}\left[  x_{1},y^{n-1}\right]  ,\mu_{n}\left[  x_{2},y^{n-1}\right]
,\ldots,\mu_{n}\left[  x_{q_{\ast}},y^{n-1}\right]  \right] \nonumber\\
&  =\mu_{n}^{\left(  q_{\ast}\right)  }\left[  \mu_{n}^{\left(  \ell\right)
}\left[  x_{1},x_{2},\ldots,x_{q_{_{\ast}}}\right]  ,y^{q_{\ast}\left(
n-1\right)  }\right]  =\mu_{n}^{\left(  q_{\ast}\right)  }\left[
x_{0},y^{q_{\ast}\left(  n-1\right)  }\right]  .
\end{align}

\textbf{2)} Insert into the formula obtained above $x_{0}=y$, then use
(\ref{xx}) to get (\ref{yy}).
\end{proof}

Finite polyadic fields $\mathbb{F}_{\left(  m,n\right)  }^{\left[  a,b\right]
}\left(  q\right)  $ having $n$-admissible reduced order $q_{\ast}=q_{\ast
}^{adm}=\ell\left(  n-1\right)  +1$ ($\ell\in\mathbb{N}$) (underlined in
\textsc{Table \ref{T2}}) are closest to the binary finite fields
$\mathbb{F}\left(  p\right)  $ in their general properties: they are
half-derived, while if they contain a zero, they are fully derived. If
$q_{\ast}\neq q_{\ast}^{adm}$ , then $\mathbb{F}_{\left(  m,n\right)
}^{\left[  a,b\right]  }\left(  q\right)  $ can be nonunital or contain more
than one unit (subscripts in \textsc{Table \ref{T2}}).

\begin{assertion}
The finite fields $\mathbb{F}_{\left(  m,n\right)  }^{\left[  a,b\right]
}\left(  q\right)  $ of $n$-admissible reduced order $q_{\ast}=q_{\ast}^{adm}$
cannot have more than one unit and cannot be zeroless-nonunital.
\end{assertion}

\begin{assertion}
If $q_{\ast}\neq q_{\ast}^{adm}$, and $\mathbb{F}_{\left(  m,n\right)
}^{\left[  a,b\right]  }\left(  q\right)  $ is unital zeroless, then the
reduced order $q_{\ast}$ is the product of the idempotence polyadic field
order $\lambda_{p}^{\left[  a,b\right]  }=\operatorname*{ord}\mathbb{F}%
_{\left(  m,n\right)  }^{\left[  a,b\right]  }\left(  q\right)  $ and the
number of units $\kappa_{e}$ (if $a\nmid b$ and $n\geq3$)%
\begin{equation}
q_{\ast}=\lambda_{p}^{\left[  a,b\right]  }\kappa_{e}.
\end{equation}

\end{assertion}

Let us consider the structure of the multiplicative group $\mathcal{G}%
_{n}^{\left[  a,b\right]  }\left(  q_{\ast}\right)  $ of $\mathbb{F}_{\left(
m,n\right)  }^{\left[  a,b\right]  }\left(  q\right)  $ in more detail. Some
properties of commutative cyclic $n$-ary groups were considered for particular
relations between orders and arity. Here we have: 1) more parameters and
different relations between these, the arity and order; 2) the $\left(
m,n\right)  $-field under consideration, which leads to additional
restrictions. In such a way exotic polyadic groups and fields arise which have
unusual properties that have not been studied before.

\begin{definition}
An element $x_{prim}\in\mathcal{G}_{n}^{\left[  a,b\right]  }\left(  q_{\ast
}\right)  $ is called $n$-\textit{ary primitive}, if its idempotence order is%
\begin{equation}
\lambda_{p}=\operatorname*{ord}x_{prim}=q_{\ast}. \label{lq}%
\end{equation}

\end{definition}

Then, all $\lambda_{p}$ polyadic powers $x_{prim}^{\left\langle 1\right\rangle
_{\times n}},x_{prim}^{\left\langle 2\right\rangle _{\times n}},\ldots
,x_{prim}^{\left\langle q_{\ast}\right\rangle _{\times n}}\equiv x_{prim}$
generate other distinct elements, and so $\mathcal{G}_{n}^{\left[  a,b\right]
}\left(  q_{\ast}\right)  $ is a finite cyclic $n$-ary group generated by
$x_{prim}$, i.e. $\mathcal{G}_{n}^{\left[  a,b\right]  }\left(  q_{\ast
}\right)  =\left\langle \left\{  x_{prim}^{\left\langle i\right\rangle
_{\times n}}\right\}  \mid\mu_{n}\right\rangle $. We denote a number primitive
elements in $\mathbb{F}_{\left(  m,n\right)  }^{\left[  a,b\right]  }\left(
q\right)  $ by $\kappa_{prim}$.

\begin{assertion}
For zeroless $\mathbb{F}_{\left(  m,n\right)  }^{\left[  a,b\right]  }\left(
q\right)  $ and prime order $q=p$, we have $\lambda_{p}^{\left[  a,b\right]
}=q$, and $\mathcal{G}_{n}^{\left[  a,b\right]  }\left(  q\right)  $ is
indecomposable ($n\geq3$).
\end{assertion}

\begin{example}
The smallest $3$-admissible zeroless polyadic field is $\mathbb{F}_{\left(
4,3\right)  }^{\left[  2,3\right]  }\left(  3\right)  $ with the unit $e=8$
and two $3$-ary primitive elements $2,5$ having $3$-idempotence order
$\operatorname*{ord}2=\operatorname*{ord}5=3$, so $\kappa_{prim}=2$ , because%
\begin{equation}
2^{\left\langle 1\right\rangle _{\times3}}=8,\ \ 2^{\left\langle
2\right\rangle _{\times3}}=5,\ \ 2^{\left\langle 3\right\rangle _{\times3}%
}=2,\ \ \ \ \ 5^{\left\langle 1\right\rangle _{\times3}}=8,\ \ 5^{\left\langle
2\right\rangle _{\times3}}=2,\ \ 5^{\left\langle 3\right\rangle _{\times3}}=5,
\end{equation}
and therefore $\mathcal{G}_{3}^{\left[  2,3\right]  }\left(  3\right)  $ is a
cyclic indecomposable $3$-ary group.
\end{example}

\begin{assertion}
If $\mathbb{F}_{\left(  m,n\right)  }^{\left[  a,b\right]  }\left(  q\right)
$ is zeroless-nonunital, then every element is $n$-ary primitive,
$\kappa_{prim}=q$, also $\lambda_{p}^{\left[  a,b\right]  }=q$ (the order $q$
can be not prime), and $\mathcal{G}_{n}^{\left[  a,b\right]  }\left(
q\right)  $ is a indecomposable commutative cyclic $n$-ary group without
identity ($n\geq3$).
\end{assertion}

\begin{example}
The $\left(  10,7\right)  $-field $\mathbb{F}_{\left(  10,7\right)  }^{\left[
5,9\right]  }\left(  9\right)  $ is zeroless-nonunital, each element (has
$\lambda_{p}=9$) is primitive and generates the whole field, and therefore
$\kappa_{prim}=9$, thus the $7$-ary multiplicative group $\mathcal{G}%
_{7}^{\left[  5,9\right]  }\left(  9\right)  $ is indecomposable and without identity.
\end{example}

The structure of $\mathcal{G}_{n}^{\left[  a,b\right]  }\left(  q_{\ast
}\right)  $ can be extremely nontrivial and may have no analogs in the binary case.

\begin{assertion}
\label{as-gg}If there exists more than one unit, then:

\begin{enumerate}
\item If $\mathcal{G}_{n}^{\left[  a,b\right]  }\left(  q_{\ast}\right)  $ can
be decomposed on its $n$-ary subroups, the number of units $\kappa_{e}$
coincides with the number of its \textsl{cyclic} $n$-ary subgroups
$\mathcal{G}_{n}^{\left[  a,b\right]  }\left(  q_{\ast}\right)  =\mathcal{G}%
_{1}\cup\mathcal{G}_{2}\ldots\cup\mathcal{G}_{k_{e}}$ which \textsl{do not}
intersect $\mathcal{G}_{i}\cap\mathcal{G}_{j}=\varnothing$, $i,j=i=1,\ldots
,\kappa_{e}$, $i\neq j$.

\item If a zero exists, then each $\mathcal{G}_{i}$ has its own unit $e_{i}$,
$i=1,\ldots,\kappa_{e}$.

\item In the zeroless case $\mathcal{G}_{n}^{\left[  a,b\right]  }\left(
q\right)  =\mathcal{G}_{1}\cup\mathcal{G}_{2}\ldots\cup\mathcal{G}_{k_{e}}\cup
E\left(  \mathcal{G}\right)  $, where $E\left(  \mathcal{G}\right)  =\left\{
e_{i}\right\}  $ is the split-off subgroup of units.
\end{enumerate}
\end{assertion}

\begin{example}
\textbf{1)} In the $\left(  9,3\right)  $-field $\mathbb{F}_{\left(
9,3\right)  }^{\left[  5,8\right]  }\left(  7\right)  $ there is a single zero
$z=21\equiv21_{z}$ and two units $e_{1}=13\equiv13_{e}$, $e_{2}=29\equiv
29_{e}$, and so its multiplicative $3$-ary group $\mathcal{G}_{3}^{\left[
5,8\right]  }\left(  6\right)  =\left\{  5,13,29,37,45,53\right\}  $ consists
of two \textsl{nonintersecting} (which is not possible in the binary case)
$3$-ary cyclic subgroups $\mathcal{G}_{1}=\left\{  5,13_{e},45\right\}  $ and
$\mathcal{G}_{2}=\left\{  29_{e},37,53\right\}  $ (for both $\lambda_{p}=3$)%
\begin{align}
\mathcal{G}_{1}  &  =\left\{  5^{\left\langle 1\right\rangle _{\times3}%
}=13_{e},5^{\left\langle 2\right\rangle _{\times3}}=45,5^{\left\langle
3\right\rangle _{\times3}}=5\right\}  ,\ \ \bar{5}=45,\overline{45}%
=5,\nonumber\\
\mathcal{G}_{2}  &  =\left\{  37^{\left\langle 1\right\rangle _{\times3}%
}=29_{e},37^{\left\langle 2\right\rangle _{\times3}}=53,37^{\left\langle
3\right\rangle _{\times3}}=37\right\}  ,\ \ \overline{37}=53,\overline
{53}=37.\nonumber
\end{align}
All nonunital elements in $\mathcal{G}_{3}^{\left[  5,8\right]  }\left(
6\right)  $ are (polyadic) $1$-reflections, because $5^{\left\langle
1\right\rangle _{\times3}}=45^{\left\langle 1\right\rangle _{\times3}}=13_{e}$
and $37^{\left\langle 1\right\rangle _{\times3}}=53^{\left\langle
1\right\rangle _{\times3}}=29_{e}$, and so the subgroup of units $E\left(
\mathcal{G}\right)  =\left\{  13_{e},29_{e}\right\}  $ is unsplit $E\left(
\mathcal{G}\right)  \cap\mathcal{G}_{1,2}\neq\varnothing$.

\textbf{2)} For the zeroless $\mathbb{F}_{\left(  9,3\right)  }^{\left[
7,8\right]  }\left(  8\right)  $, its multiplicative $3$-group $\mathcal{G}%
_{3}^{\left[  5,8\right]  }\left(  6\right)  =\left\{
7,15,23,31,39,47,55,63\right\}  $ has two units $e_{1}=31\equiv31_{e}$,
$e_{2}=63\equiv63_{e}$, and it splits into two nonintersecting nonunital
cyclic $3$-subgroups ($\lambda_{p}=4$ and $\lambda_{p}=2$) and the subgroup of
units%
\begin{align*}
&  \mathcal{G}_{1} =\left\{  7^{\left\langle 1\right\rangle _{\times3}%
}=23,7^{\left\langle 2\right\rangle _{\times3}}=39,7^{\left\langle
3\right\rangle _{\times3}}=55,7^{\left\langle 4\right\rangle _{\times3}%
}=4\right\}  ,\ \bar{7}=55,\overline{55}=7,\overline{23}=39,\overline
{39}=23,\\
& \mathcal{G}_{2} =\left\{  15^{\left\langle 1\right\rangle _{\times3}%
}=47,15^{\left\langle 2\right\rangle _{\times3}}=15\right\}  ,\ \ \ \overline
{15}=47,\overline{47}=15,\\
&  E\left(  \mathcal{G}\right)  =\left\{  31_{e},63_{e}\right\}  .
\end{align*}
There are no $\ell_{\mu}$-reflections, and so $E\left(  \mathcal{G}\right)  $
splits out $E\left(  \mathcal{G}\right)  \cap\mathcal{G}_{1,2}=\varnothing$.
\end{example}

If all elements are units $E\left(  \mathcal{G}\right)  =\mathcal{G}%
_{n}^{\left[  a,b\right]  }\left(  q\right)  $, then, obviously, this group is
$1$-idempotent, and $\lambda_{p}=1$.

\begin{assertion}
If $\mathbb{F}_{\left(  m,n\right)  }^{\left[  a,b\right]  }\left(  q\right)
$ is zeroless-nonunital, then there no $n$-ary cyclic subgroups in
$\mathcal{G}_{n}^{\left[  a,b\right]  }\left(  q\right)  $.
\end{assertion}

The subfield structure of $\mathbb{F}_{\left(  m,n\right)  }^{\left[
a,b\right]  }\left(  q\right)  $ can coincide with the corresponding subgroup
structure of the multiplicative $n$-ary group $\mathcal{G}_{n}^{\left[
a,b\right]  }\left(  q_{\ast}\right)  $, only if its additive $m$-ary group
has the same subgroup structure. However, additive $m$-ary groups of all
polyadic fields $\mathbb{F}_{\left(  m,n\right)  }^{\left[  a,b\right]
}\left(  q\right)  $ have the same structure: they are cyclic and have no
proper $m$-ary subgroups, each element generates all other elements, i.e. it
is a primitive root. Therefore, we arrive at

\begin{theorem}
The polyadic field $\mathbb{F}_{\left(  m,n\right)  }^{\left[  a,b\right]
}\left(  q\right)  $, being isomorphic to the $\left(  m,n\right)  $-field of
polyadic integer numbers $\mathbb{Z}_{\left(  m,n\right)  }^{\left[
a,b\right]  }\left(  q\right)  $, has no any proper subfield.
\end{theorem}

In this sense, $\mathbb{F}_{\left(  m,n\right)  }^{\left[  a,b\right]
}\left(  q\right)  $ can be named a\textit{ prime} \textit{polyadic field}.

\section{\textsc{Conclusions}}

Recall that any binary finite field has an order which is a power of a prime
number $q=p^{r}$ (its characteristic), and all such fields are isomorphic and
contain a prime subfield $GF\left(  p\right)  $ of order $p$ which is
isomorphic to the congruence (residue) class field $\mathbb{Z}\diagup
p\mathbb{Z}$ \cite{lid/nie}.

\begin{conjecture}
A finite $\left(  m,n\right)  $-field (with $m>n$) should contain a minimal
subfield which is isomorphic to one of the prime polyadic fields constructed
above, and therefore $\mathbb{F}_{\left(  m,n\right)  }^{\left[  a,b\right]
}\left(  q\right)  $ can be interpreted as a polyadic analog of $GF\left(
p\right)  $.
\end{conjecture}

\mbox{} \bigskip

\subsection*{\textsc{Acknowledgements}}

The author would like to express his sincere thankfulness and gratitude to
Joachim Cuntz, Christopher Deninger, Mike Hewitt, Grigorij Kurinnoj, Daniel
Lenz, Jim Stasheff, Alexander Voronov, and Wend Werner for fruitful discussions.

\newpage

\mbox{}

\appendix
\setcounter{section}{1}
\renewcommand{\theequation}{{\color{Magenta}{\sf\Alph{section}.\arabic{equation}}}}
\setcounter{secnumdepth}{0} \bigskip\numberwithin{theorem}{section} \renewcommand{\thetheorem}{{\color{Brown}\bf\arabic{section}.\arabic{theorem}}}

\section{\textsc{Appendix.} Multiplicative properties of exotic finite
polyadic fields}

Here we list concrete examples of finite polyadic fields which have properties
that are not possible in the binary case (see \textsc{Table \ref{T2}}). Only
the multiplication of fields will be shown, because their additive part is
huge (many pages) for higher arities, and does not carry so much distinctive information.

\textbf{1)} The first exotic finite polyadic field which has a number of
elements which is not a prime number, or prime power (as it should be for a
finite binary field) is $\mathbb{F}_{\left(  7,3\right)  }^{\left[
5,6\right]  }\left(  6\right)  $, which consists of $6$ elements $\left\{
5,11,17,23,29,35\right\}  $, $q=6$, It is zeroless and contains two units
$\left\{  17,35\right\}  \equiv\left\{  17_{e},35_{e}\right\}  $, $\kappa
_{e}=2$, and each element has the idempotence polyadic order $\lambda_{p}=3$,
i.e. $\mu_{3}\left[  x^{7}\right]  =x,\ \forall x\in\mathbb{F}_{\left(
7,3\right)  }^{\left[  5,6\right]  }\left(  6\right)  $. The multiplication is%
\begin{align*}
\mu_{3}\left[  5,5,5\right]   &  =17,\mu_{3}\left[  5,5,11\right]  =23,\mu
_{3}\left[  5,5,17\right]  =29,\mu_{3}\left[  5,5,23\right]  =35,\mu
_{3}\left[  5,5,29\right]  =5,\\
\mu_{3}\left[  5,5,35\right]   &  =11,\mu_{3}\left[  5,11,11\right]
=29,\mu_{3}\left[  5,11,17\right]  =35,\mu_{3}\left[  5,11,23\right]
=5,\mu_{3}\left[  5,11,29\right]  =11,\\
\mu_{3}\left[  5,11,35\right]   &  =17,\mu_{3}\left[  5,17,17\right]
=5,\mu_{3}\left[  5,17,23\right]  =11,\mu_{3}\left[  5,17,29\right]
=17,\mu_{3}\left[  5,17,35\right]  =23,\\
\mu_{3}\left[  5,23,23\right]   &  =17,\mu_{3}\left[  5,23,29\right]
=23,\mu_{3}\left[  5,23,35\right]  =29,\mu_{3}\left[  5,29,29\right]
=29,\mu_{3}\left[  5,29,35\right]  =35,\\
\mu_{3}\left[  5,35,35\right]   &  =5,\mu_{3}\left[  11,11,11\right]
=35,\mu_{3}\left[  11,11,17\right]  =5,\mu_{3}\left[  11,11,23\right]
=11,\mu_{3}\left[  11,11,29\right]  =17,\\
\mu_{3}\left[  11,11,35\right]   &  =23,\mu_{3}\left[  11,11,17\right]
=5,\mu_{3}\left[  11,17,17\right]  =11,\mu_{3}\left[  11,17,23\right]
=17,\mu_{3}\left[  11,17,29\right]  =23,\\
\mu_{3}\left[  11,17,35\right]   &  =29,\mu_{3}\left[  11,17,23\right]
=17,\mu_{3}\left[  11,17,29\right]  =23,\mu_{3}\left[  11,17,35\right]
=29,\mu_{3}\left[  11,23,23\right]  =23,\\
\mu_{3}\left[  11,23,29\right]   &  =29,\mu_{3}\left[  11,23,23\right]
=23,\mu_{3}\left[  11,23,35\right]  =35,\mu_{3}\left[  11,29,29\right]
=35,\mu_{3}\left[  11,29,35\right]  =5,\\
\mu_{3}\left[  11,35,35\right]   &  =11,\mu_{3}\left[  17,17,17\right]
=17,\mu_{3}\left[  17,17,23\right]  =23,\mu_{3}\left[  17,17,29\right]
=29,\mu_{3}\left[  17,17,35\right]  =35,\\
\mu_{3}\left[  17,23,23\right]   &  =29,\mu_{3}\left[  17,23,29\right]
=35,\mu_{3}\left[  17,29,29\right]  =5,\mu_{3}\left[  17,29,35\right]
=11,\mu_{3}\left[  17,35,35\right]  =17,\\
\mu_{3}\left[  23,23,23\right]   &  =35,\mu_{3}\left[  23,23,29\right]
=5,\mu_{3}\left[  23,23,35\right]  =11,\mu_{3}\left[  23,29,29\right]
=11,\mu_{3}\left[  23,29,35\right]  =17,\\
\mu_{3}\left[  23,35,35\right]   &  =23,\mu_{3}\left[  29,29,29\right]
=17,\mu_{3}\left[  29,29,35\right]  =23,\mu_{3}\left[  29,35,35\right]
=29,\mu_{3}\left[  35,35,35\right]  =35.
\end{align*}
The multiplicative querelements are $\bar{5}=29,\overline{29}=5,\overline
{11}=23,\overline{23}=11$. Because%
\begin{align}
5^{\left\langle 1\right\rangle _{\times3}}  &  =17_{e},\ \ 5^{\left\langle
2\right\rangle _{\times3}}=29,\ \ 5^{\left\langle 3\right\rangle _{\times3}%
}=5,\ \ 29^{\left\langle 1\right\rangle _{\times3}}=17_{e}%
,\ \ 29^{\left\langle 2\right\rangle _{\times3}}=5,\ \ 29^{\left\langle
3\right\rangle _{\times3}}=29,\\
11^{\left\langle 1\right\rangle _{\times3}}  &  =35_{e},\ \ 11^{\left\langle
2\right\rangle _{\times3}}=23,\ \ 11^{\left\langle 3\right\rangle _{\times3}%
}=11,\ \ 23^{\left\langle 1\right\rangle _{\times3}}=35_{e}%
,\ \ 23^{\left\langle 2\right\rangle _{\times3}}=11,\ \ 23^{\left\langle
3\right\rangle _{\times3}}=23,
\end{align}
the multiplicative $3$-ary group $\mathcal{G}_{\left(  7,3\right)  }^{\left[
5,6\right]  }\left(  6\right)  $ consists of two \textsl{nonintersecting}
cyclic $3$-ary subgroups%
\begin{align}
\mathcal{G}_{\left(  7,3\right)  }^{\left[  5,6\right]  }\left(  6\right)   &
=\mathcal{G}_{1}\cup\mathcal{G}_{2},\ \ \ \mathcal{G}_{1}\cap\mathcal{G}%
_{2}=\varnothing,\label{gg}\\
\mathcal{G}_{1}  &  =\left\{  5,17_{e},29\right\}  ,\\
\mathcal{G}_{2}  &  =\left\{  11,23,35_{e}\right\}  ,
\end{align}
which is impossible for binary subgroups, as these always intersect in the
identity of the binary group.

\textbf{2)} The finite polyadic field $\mathbb{F}_{\left(  7,3\right)
}^{\left[  5,6\right]  }\left(  4\right)  =\left\{  \left\{
5,11,17,23\right\}  \mid\nu_{7},\mu_{3}\right\}  $ which has the same arity
shape as above, but with order $4$, has the exotic property that \textsl{all
elements }are units, which follows from its multiplication%
\begin{align*}
\mu_{3}\left[  5,5,5\right]   &  =5,\mu_{3}\left[  5,5,11\right]  =11,\mu
_{3}\left[  5,5,17\right]  =17,\mu_{3}\left[  5,5,23\right]  =23,\mu
_{3}\left[  5,11,11\right]  =5,\\
\mu_{3}\left[  5,11,17\right]   &  =23,\mu_{3}\left[  5,11,23\right]
=17,\mu_{3}\left[  5,17,17\right]  =5,\mu_{3}\left[  5,17,23\right]
=11,\mu_{3}\left[  5,23,23\right]  =5,\\
\mu_{3}\left[  11,11,11\right]   &  =11,\mu_{3}\left[  11,11,17\right]
=17,\mu_{3}\left[  11,11,23\right]  =23,\mu_{3}\left[  11,17,17\right]
=11,\mu_{3}\left[  11,17,23\right]  =5,\\
\mu_{3}\left[  11,23,23\right]   &  =11,\mu_{3}\left[  17,17,17\right]
=17,\mu_{3}\left[  17,17,23\right]  =23,\mu_{3}\left[  17,23,23\right]
=17,\mu_{3}\left[  23,23,23\right]  =23.
\end{align*}

\textbf{3)} Next we show by construction, that (as opposed to the case of
binary finite fields) there exist \textsl{non-isomorphic} finite polyadic
fields of the same order and arity shape. Indeed, consider these two $\left(
9,3\right)  $-fields of order $2$, that are $\mathbb{F}_{\left(  9,3\right)
}^{\left[  3,8\right]  }\left(  2\right)  $ and $\mathbb{F}_{\left(
9,3\right)  }^{\left[  7,8\right]  }\left(  2\right)  $. The first is
zeroless-nonunital, while the second is zeroless with two units, i.e. all
elements are units. The multiplication of $\mathbb{F}_{\left(  9,3\right)
}^{\left[  3,8\right]  }\left(  2\right)  $ is%
\[
\mu_{3}\left[  3,3,3\right]  =11,\ \mu_{3}\left[  3,3,11\right]  =3,\ \mu
_{3}\left[  3,11,11\right]  =11,\ \mu_{3}\left[  11,11,11\right]  =3,
\]
having the multiplicative querelements $\bar{3}=11$, $\overline{11}=3$. For
$\mathbb{F}_{\left(  9,3\right)  }^{\left[  7,8\right]  }\left(  2\right)  $
we get the $3$-group of units%
\[
\mu_{3}\left[  7,7,7\right]  =7,\ \mu_{3}\left[  7,7,15\right]  =15,\ \mu
_{3}\left[  7,15,15\right]  =7,\ \mu_{3}\left[  15,15,15\right]  =15.
\]
They have different idempotence polyadic orders $\operatorname*{ord}%
\mathbb{F}_{\left(  9,3\right)  }^{\left[  3,8\right]  }\left(  2\right)  =2$
and $\operatorname*{ord}\mathbb{F}_{\left(  9,3\right)  }^{\left[  7,8\right]
}\left(  2\right)  =1$. Despite their additive $m$-ary groups being
isomorphic, it follows from the above multiplicative structure, that it is not
possible to construct an isomorphism between the fields themselves.

\textbf{4)} The smallest exotic finite polyadic field with more than one unit
is $\mathbb{F}_{\left(  4,3\right)  }^{\left[  2,3\right]  }\left(  5\right)
=\left\{  \left\{  2,5,8,11,14\right\}  \mid\nu_{4},\mu_{3}\right\}  $ of
order $5$ with two units $\left\{  11,14\right\}  \equiv\left\{  11_{e}%
,14_{e}\right\}  $ and the zero $5\equiv5_{z}$. The presence of zero allows us
to define the polyadic characteristic (\ref{ez}) which is $3$ (see
\textsc{Table \ref{T0}}), because the 3rd additive power of all elements is
equal to zero%
\begin{equation}
2^{\left\langle 3\right\rangle _{+4}}=8^{\left\langle 3\right\rangle _{+4}%
}=11_{e}^{\left\langle 3\right\rangle _{+4}}=14_{e}^{\left\langle
3\right\rangle _{+4}}=5_{z}.
\end{equation}
The additive querelements are%
\begin{equation}
\tilde{2}=11_{e},\ \ \tilde{8}=14_{e},\ \ \widetilde{11}_{e}=8,\ \ \widetilde
{14}_{e}=2.
\end{equation}
The idempotence polyadic order is $\operatorname*{ord}\mathbb{F}_{\left(
4,3\right)  }^{\left[  2,3\right]  }\left(  5\right)  =2$, because for nonunit
and nonzero elements%
\begin{equation}
2^{\left\langle 2\right\rangle _{\times3}}=2,\ \ 8^{\left\langle
2\right\rangle _{\times3}}=8,
\end{equation}
and their multiplicative querelements are $\bar{2}=8$, $\bar{8}=2$. The
multiplication is given by the cyclic $3$-ary group $\mathcal{G}_{3}^{\left[
2,3\right]  }\left(  4\right)  =\left\{  \left\{  2,8,11,14\right\}  \mid
\mu_{3}\right\}  $ as:%
\begin{align*}
\mu_{3}\left[  2,2,2\right]   &  =8,\mu_{3}\left[  2,2,8\right]  =2,\mu
_{3}\left[  2,2,11\right]  =14,\mu_{3}\left[  2,2,14\right]  =11,\mu
_{3}\left[  2,8,8\right]  =8,\\
\mu_{3}\left[  2,8,11\right]   &  =11,\mu_{3}\left[  2,8,14\right]
=14,\mu_{3}\left[  2,11,11\right]  =2,\mu_{3}\left[  2,11,14\right]
=8,\mu_{3}\left[  2,14,14\right]  =2,\\
\mu_{3}\left[  8,8,8\right]   &  =2,\mu_{3}\left[  8,8,11\right]  =14,\mu
_{3}\left[  8,8,14\right]  =11,\mu_{3}\left[  8,11,11\right]  =8,\mu
_{3}\left[  8,11,14\right]  =2,\\
\mu_{3}\left[  8,14,14\right]   &  =8,\mu_{3}\left[  11,11,11\right]
=11,\mu_{3}\left[  11,11,14\right]  =14,\mu_{3}\left[  11,14,14\right]
=11,\mu_{3}\left[  14,14,14\right]  =14.
\end{align*}

We observe that, despite having two units, the cyclic $3$-ary group
$\mathcal{G}_{3}^{\left[  2,3\right]  }\left(  4\right)  $ has no
decomposition into nonintersecting cyclic $3$-ary subgroups, as in (\ref{gg}).
One of the reasons is that the polyadic field $\mathbb{F}_{\left(  7,3\right)
}^{\left[  5,6\right]  }\left(  6\right)  $ is zeroless, while $\mathbb{F}%
_{\left(  4,3\right)  }^{\left[  2,3\right]  }\left(  5\right)  $ has a zero
(see \textbf{Assertion \ref{as-gg}}).

\newpage
\mbox{}

\mbox{}
\bigskip

\listoftables

\end{document}